\documentclass{article}
\usepackage{amsmath}
\usepackage{amsthm}
\usepackage{amssymb} 
\usepackage{eucal}
\usepackage{esint}
\usepackage{mathrsfs}
\usepackage{graphicx,graphics}
\numberwithin{equation}{section}
\usepackage{graphicx,graphics}
\usepackage{pdfsync}
\usepackage{enumerate}
\usepackage{pdfsync}
\usepackage[colorlinks=true, linkcolor=blue]{hyperref} 
\usepackage{hyperref}
\newtheorem{theorem}{Theorem}[section]
\newtheorem{lemma}[theorem]{Lemma}

\newtheorem{definition}[theorem]{Definition}
\newtheorem{remark}[theorem]{Remark}
\newtheorem{proposition}[theorem]{Proposition}

\DeclareMathOperator{\dive}{div}

\DeclareMathOperator{\supp}{supp}

\DeclareMathOperator{\tr}{tr}
\DeclareMathOperator{\per}{per}
\DeclareMathOperator{\adj}{adj}
\DeclareMathOperator*{\essinf}{ess-inf}
\DeclareMathOperator{\Cof}{Cof}
\DeclareMathOperator{\se}{se}
\DeclareMathOperator{\vse}{vse}
\def\LL{\mathbb{L}}
\def\RR{\mathbb{R}}

\def\F{\mathscr{F}}
\def\H{\mathscr{H}}
\def\L{\mathscr{L}}
\def\Om{\Omega}

\def\ep{\varepsilon}

\title{Homogenization of weakly coercive integral functionals in three-dimensional elasticity}
\author{
M.~Briane\footnote{IRMAR \& INSA Rennes, mbriane@insa-rennes.fr},
A.~Pallares-Mart\'in\footnote{Dpto. de Ecuaciones Diferenciales y An\'alisis Num\'erico, Universidad de Sevilla, ajpallares@us.es} 
}

\begin{document}
\maketitle

\begin{abstract}
This paper deals with the homogenization through $\Gamma$-convergence of weakly coercive integral energies with the oscillating density $\LL(x/\ep)\nabla v:\nabla v$ in three-dimensional elasticity. The energies are weakly coercive in the sense where the classical functional coercivity satisfied by the periodic tensor $\LL$ (using smooth test functions $v$ with compact support in $\RR^3$) which reads as $\Lambda(\LL)>0$, is replaced by the relaxed condition $\Lambda(\LL)\geq 0$. Surprisingly, we prove that contrary to the two-dimensional case of~\cite{BF01} which seems {\em a priori} more constrained, the homogenized tensor $\LL^0$ remains strongly elliptic, or equivalently $\Lambda(\LL^0)>0$, for any tensor $\LL=\LL(y_1)$ satisfying $\LL(y)M:M+D:{\rm Cof}(M)\geq 0$, a.e. $y\in\RR^3$, $\forall\,M\in\RR^{3\times 3}$, for some matrix $D\in\RR^{3\times 3}$ (which implies $\Lambda(\mathbb{L})\geq 0$), and the periodic functional coercivity (using smooth test functions $v$ with periodic gradients) which reads as $\Lambda_{\rm per}(\LL)>0$. Moreover, we derive the loss of strong ellipticity for the homogenized tensor using a rank-two lamination, which justifies by $\Gamma$-convergence the formal procedure of~\cite{G01}.
\end{abstract}


\section{Introduction}
In this paper, for a bounded domain $\Omega$ of $\RR^3$ and for a periodic symmetric tensor-valued function $\LL=\LL(y)$, we study the homogenization of the elasticity energy
\begin{equation}\label{elasener}
v\in H^1_0(\Om;\RR^3)\mapsto\int_\Om\LL(x/\ep)\nabla v\cdot\nabla v\,dx\quad\mbox{as }\ep\to 0,
\end{equation}
especially when the tensor $\mathbb{L}$ is weakly coercive (see below).
It is shown in \cite{San,Francfort1} that for any periodic symmetric tensor-valued function $\LL=\LL(y)$ satisfying the functional coercivity, {\em i.e.}
\begin{equation}\label{FuncCoerc}
\Lambda(\LL):=\inf \left\{ \int_{\RR^3}\LL \nabla v : \nabla v\,dy,\ v\in C_c^\infty(\RR^3;\mathbb{R}^3),\ \int_{\RR^3}|\nabla v|^2\,dy = 1\right\} > 0,
\end{equation}
and for any $f\in H^{-1}(\Om;\RR^3)$, the elasticity system
\begin{equation}\label{syselas}
\left\{\begin{array}{cl}
-\,{\rm div}\big(\LL(x/\ep)\nabla u^\ep\big)=f & \mbox{in }\Om
\\*[1.mm]
u^\ep=0 & \mbox{on }\partial\Om,
\end{array}\right.
\end{equation}
H-converges as $\ep\to 0$ in the sense of Murat-Tartar \cite{Mur} to the elasticity system with the so-called homogenized tensor $\LL^0$ defined by
\begin{equation}\label{L0}
\LL^0M:M := \inf\left\{ \int_{Y_3} \LL(M+\nabla v):(M+\nabla v)\,dy,\ v\in H^1_{\per}(Y_3;\mathbb{R}^3) \right\}
\quad\mbox{for }M\in \mathbb{R}^{3\times 3}.
\end{equation}
Equivalently, under the functional coercivity \eqref{FuncCoerc} the energy \eqref{elasener} $\Gamma$-converges for the weak topology of $H^1_0(\Om;\RR^3)$ (see Definition~\ref{defGcon}) to the functional
\begin{equation}\label{GalimL0}
v\in H^1_0(\Om;\RR^3)\mapsto\int_\Om\LL^0\nabla v:\nabla v\,dx.
\end{equation}
The functional coercivity \eqref{FuncCoerc}, which is a nonlocal condition satisfied by the symmetric tensor $\LL$, is implied by the very strong ellipticity, {\em i.e.} the local condition
\begin{equation}\label{vseIntro}
\alpha_{\rm vse}(\LL):=\essinf_{y\in\RR^3}\big( \min\{ \LL(y)M:M, \ M\in \mathbb{R}^{3\times 3}_s,\ |M|=1\} \big)>0,
 \end{equation}
 and the converse is not true in general. Moreover, condition \eqref{FuncCoerc} implies the strong ellipticity, {\em i.e.}
\begin{equation}\label{seIntro}
\alpha_{\rm se}(\LL):=\essinf_{y\in\RR^3}\big( \min\{ \LL(y)(a\otimes b):(a\otimes b), \ a,b\in\RR^3,\ |a|=|b|=1 \}\big) > 0,
\end{equation}
but contrary to the scalar case, the converse is not true in general.
\par
Here, we focus on the case where the tensor $\LL$ is weakly coercive, {\em i.e.} relaxing the condition $\Lambda(\LL)>0$ by $\Lambda(\LL)\geq 0$.
In this case the homogenization of the elasticity system \eqref{syselas} associated with the energy \eqref{elasener} is badly posed in general, since one has no {\em a priori} $L^2$-bound on the stress tensor $\nabla u^\ep$ (assuming the existence of a solution $u^\ep$ to the elasticity system \eqref{syselas}) due to the loss of coercivity. However, it was shown by Geymonat {\em et al.} \cite{GMT} that the previous $\Gamma$-convergence result still holds when $\Lambda(\LL)\geq 0$, under the extra condition of periodic functional coercivity, {\em i.e.}
\begin{equation}\label{percoer}
\Lambda_{\per}(\LL):= \inf\left\{ \int_{Y_3} \LL \nabla v: \nabla v\,dy,\ v\in H^1_{\per}(Y_3;\mathbb{R}^3),\ \int_{Y_3} |\nabla v|^2 \, dy = 1\right\}>0.
\end{equation}
Furthermore, using the Murat-Tartar $1^*$-convergence for tensors which depend only on one direction (see \cite{Mur} in the conductivity case, see \cite[Section~3]{G01} and \cite[Lemma~3.1]{BF01} in the elasticity case) Guti\'errez \cite[Proposition~1]{G01} derived in two and three dimensions a $1$-periodic rank-one laminate with two isotropic phases whose tensor is
\begin{equation}\label{Gr1lam}
\LL_1(y_1)=\chi_1(y_1)\,\LL_a+\big(1-\chi(y_1)\big)\,\LL_b\quad\mbox{for }y_1\in\RR,
\end{equation}
which is strongly elliptic, {\em i.e.} $\alpha_{\rm se}(\LL)>0$, and weakly coercive, {\em i.e.} $\Lambda(\LL)\geq 0$, but such that the homogenized tensor~$\LL^0$ (in fact the homogenized tensor induced by $1^*$-convergence which is shown to agree with $\LL^0$ in the step~4 of the proof of Theorem~\ref{ThRankTwoLossStrEll}) is not strongly elliptic, {\em i.e.} $\alpha_{\se}(\LL^0)=0$. However, the $1^*$-convergence process used by Guti\'errez in \cite{G01} needs to have {\em a priori} $L^2$-bounds for the sequence of deformations, which is not compatible with the weak coercivity assumption. Therefore, Guti\'errez' approach is not a H-convergence process applied to the elasticity system \eqref{syselas}. Francfort and the first author \cite{BF01} obtained in dimension two a similar loss of ellipticity through a homogenization process using the $\Gamma$-convergence approach of \cite{GMT} from a more generic (with respect to \eqref{Gr1lam}) $1$-periodic isotropic tensor $\LL=\LL(y_1)$ satisfying
\begin{equation}\label{lossell}
\Lambda(\LL)=0,\quad \Lambda_{\rm per}(\LL)>0\quad\mbox{and}\quad\alpha_{\se}(\LL^0)=0.
\end{equation}
They also showed that Guti\'errez' lamination is the only one among rank-one laminates which implies such a loss of strong ellipticity.
\par\bigskip
The aim of the paper is to extend the result of \cite{BF01} to dimension three, namely justifying the loss of ellipticity of \cite{G01} by a homogenization process. The natural idea is to find as in \cite{BF01} a $1$-periodic isotropic tensor $\LL=\LL(y_1)$ satisfying \eqref{lossell}. 
Firstly, in order to check the relaxed functional coercivity $\Lambda(\LL)\geq 0$, we apply the translation method used in \cite{BF01}, which consists in adding to the elastic energy density a suitable null lagrangian such that the following pointwise inequality holds for some matrix $D\in\mathbb{R}^{3\times 3}$:
\begin{equation}\label{transmeth}
\LL M:M + D:\Cof(M)\ge0,\quad\forall\,M\in \mathbb{R}^{3\times 3}.
\end{equation}
Note that in dimension two the translation method reduces to adding the term $d\det(M)$ with one coefficient $d$, rather than a $(3\times 3)$-matrix $D$ in dimension three.
But surprisingly, and contrary to the two-dimensional case of~\cite{BF01}, we prove (see Theorem \ref{ThNoLossFirstLaminate}) that for any $1$-periodic tensor $\LL=\LL(y_1)$, condition \eqref{transmeth} combined with $\Lambda_{\rm per}(\LL)>0$ actually implies that $\alpha_{\se}(\LL^0)>0$, making impossible the loss of ellipticity through homogenization. This specificity was already observed by Guti\'errez \cite{G01} in the particular case of isotropic two-phase rank-one laminates~\eqref{Gr1lam}, where certain regimes satisfied by the Lam\'e coefficients of the isotropic phases $\LL_a,\LL_b$ are not compatible with the desired equality $\alpha_{\se}(\LL^0)=0$.
\par
To overcome this difficulty Guti\'errez \cite{G01} considered a rank-two laminate obtained by mixing in the direction $y_2$ the homogenized tensor $\LL_1^*$ of $\mathbb{L}_1(y_1)$ defined by \eqref{Gr1lam}, with a very strongly elliptic isotropic tensor $\LL_c$. In the present context we derive a similar loss of ellipticity by rank-two lamination, but justifying it through homogenization still using a $\Gamma$-convergence procedure (see Theorem \ref{ThRankTwoLossStrEll}). However, the proof is rather delicate, since we have to choose the isotropic materials $a,b,c$ so that the $1$-periodic rank-one laminate tensor $\LL_2$ in the direction $y_2$ obtained after the first rank-one lamination of $\LL_a,\LL_b$ in the direction $y_1$, namely
\begin{equation}\label{r2lam}
\LL_2(y_2)=\chi_2(y_2)\,\LL_1^*+\big(1-\chi_2(y_2)\big)\,\LL_c\quad\mbox{for }y_2\in\RR,
\end{equation}
satisfies
\begin{equation}\label{lossell2}
\Lambda(\LL_2)\geq 0\quad\mbox{and}\quad\alpha_{\se}(\LL_2^0)=0,
\end{equation}
where $\LL_2^0$ is the homogenized tensor defined by formula \eqref{L0} with $\LL=\LL_2$. Moreover, the condition $\Lambda(\LL_2)\geq 0$ without $\Lambda_{\rm per}(\LL_2)>0$ (which seems very intricate to check) needs to extend the $\Gamma$-convergence result of \cite[Theorem~3.1(i)]{GMT}. However, Braides and the first author have proved (see Theorem~\ref{BBResult}) that the $\Gamma$-convergence result for the energy \eqref{elasener} holds true under the sole condition $\Lambda(\LL)\geq 0$.
\par\bigskip
The paper is divided in two sections. In the first section we prove the $\Gamma$-convergence result for \eqref{elasener} under the assumption $\Lambda(\mathbb{L})\ge0$, and without the condition $\Lambda_{\rm per}(\LL)>0$. The second section is devoted to the main results of the paper: In Section~\ref{SubSecRank1Lam} we prove the strong ellipticity of the homogenized tensor $\LL^0$ for any isotropic tensor $\LL=\LL(y_1)$ satisfying both the two conditions \eqref{transmeth} (which implies $\Lambda(\LL)\geq 0$) and $\Lambda_{\rm per}(\LL)>0$. In Section~\ref{SubSecRank2Lam} we show the loss ellipticity by homogenization using a suitable rank-two laminate tensor $\LL_2$ of type \eqref{r2lam}, and the $\Gamma$-convergence result under the sole condition $\Lambda(\LL_2)\geq 0$. Finally, the Appendix is devoted to the proof of Theorem~\ref{LamdaPer>0}.

\subsubsection*{Notations}
\begin{itemize}
\item The space dimension is denoted by $N\geq 2$, but most of the time it will be $N=3$.

\item $\mathbb{R}_s^{N\times N}$ denotes the set of the symmetric matrices in $\mathbb{R}^{N\times N}$.

\item $I_N$ denotes the identity matrix of $\mathbb{R}^{N\times N}$.

\item For any $M\in\RR^{N\times N}$, $M^T$ denotes the transposed of $M$, and $M^s$ denotes the symmetrized matrix of $M$.

\item $:$ denotes the Frobenius inner product in $\mathbb{R}^{N\times N}$, {\em i.e.} $M:M':=\tr(M^TM')$ for $M,M'\in\RR^{N\times N}$.

\item $\L_s(\mathbb{R}^{N\times N})$ denotes the space of the symmetric tensors $\LL$ on $\mathbb{R}^{N\times N}$ satisfying
\[
\LL M=\LL M^s\in\mathbb{R}_s^{N\times N}\quad\mbox{and}\quad \LL M:M'=\LL M':M,
\quad\forall\,M,M'\in\mathbb{R}_s^{N\times N}.
\]
In terms of the entries $\LL_{ijkl}$ of $\LL$, this is equivalent to $\LL_{ijkl}=\LL_{jikl}=\LL_{klij}$ for any $i,j,k,l\in\{1,\dots,N\}$.

\item $\mathbb{I}_s$ denotes the unit tensor of $\L_s(\mathbb{R}^{N\times N})$ defined by $\mathbb{I}_s M:=M^s$ for $M\in\RR^{N\times N}$.

\item $M_{ij}$ denotes the $(i,j)$ entry of the matrix $M\in\mathbb{R}^{N\times N}$.

\item $\tilde{M}^{ij}$ denotes the $(N\!-\!1)\times(N\!-\!1)$-matrix resulting from deleting the $i$-th row and the $j$-th column of the matrix $M\in\mathbb{R}^{N\times N}$ for $i,j\in\{1,\dots,N\}$.

\item ${\Cof} (M)$ denotes the cofactors matrix of $M\in \mathbb{R}^{N\times N}$, {\em i.e.} the matrix with entries
$({\Cof} M)_{ij} = (-1)^{i+j} \det(\tilde{M}^{ij})$ for $i,j\in\{1,\dots,N\}$.
   
\item $\adj(M)$ denotes the adjugate matrix of $M\in\mathbb{R}^{N\times N}$, {\em i.e.} $\adj(M) = (\Cof M)^T$.

\item $Y_N:=[0,1)^N$ denotes the unit cube of $\RR^N$.

\end{itemize}
Let $\LL\in L^\infty_{\per}\big(Y_N;\L_s(\mathbb{R}^{N\times N})\big)$ be a $Y_N$-periodic symmetric tensor-valued function.
In the whole paper we will use the following ellipticity constants related to the tensor $\LL$ (see \cite[Section~3]{GMT} for further details):
    \begin{itemize}
    \item $\alpha_{\se}(\LL)$ denotes the best ellipticity constant for $\LL$, {\em i.e.}
    \begin{equation*}
    \alpha_{\se}(\LL):= \essinf_{y\in Y_N}\big( \min\{ \LL(y)(a\otimes b):(a\otimes b), \ a,b\in \mathbb{R}^{N},\ |a|=|b|=1 \} \big).
    \end{equation*} 
    \item $\alpha_{\vse}(\LL)$ denotes the best constant of very strong ellipticity of $\LL$, {\em i.e.}
    \begin{equation*}
    \alpha_{\vse}(\LL):=  \essinf_{y\in Y_N}\big( \min\{ \LL(y)M:M, \ M\in \mathbb{R}^{N\times N}_s, |M|=1 \} \big).
    \end{equation*}
    \item $\Lambda(\LL)$ denotes the global functional coercivity constant for $\LL$, {\em i.e.}
    \begin{equation*}\label{DefLambda}
\Lambda (\LL) := \inf \left\{ \int_{\mathbb{R}^N} \LL \nabla v : \nabla v\,dy,\ v\in C_c^\infty(\mathbb{R}^N;\mathbb{R}^N),\ \int_{\mathbb{R}^N} |\nabla v|^2\,dy = 1 \right\}.
\end{equation*}
\item $\Lambda_{\per}(\LL)$ denotes the functional coercivity constant of $\LL$ with respect to $Y_N$-periodic deformations, {\em i.e.}
\begin{equation*}
\Lambda_{\per}(\LL) := \inf\left\{ \int_{Y_N} \LL \nabla v: \nabla v\,dy,\ v\in H^1_{\per}(Y_N;\mathbb{R}^N),\ \int_{Y_N} |\nabla v|^2 \, dy = 1\right\}.
\end{equation*}
\end{itemize}
\begin{remark}\label{rem.ell}
\noindent
\begin{itemize}
\item The very strong ellipticity implies the strong ellipticity, {\em i.e.} for any tensor $\LL$,
\[
\alpha_{\rm vse}(\LL)>0\;\Rightarrow\;\alpha_{\rm se}(\LL)>0.
\]
\item According to \cite[Theorem~3.3(i)]{GMT}, if $\alpha_{\rm se}(\LL)>0$, then the following inequalities hold:
\begin{equation}\label{inecoer}
\Lambda (\LL)\leq\Lambda_{\rm per}(\LL)\leq\alpha_{\rm se}(\LL).
\end{equation}
\item Using a Fourier transform we get that for any constant tensor $\LL_0$,
\[
\alpha_{\rm se}(\LL_0)>0\;\Leftrightarrow\;\Lambda(\LL_0)>0.
\]
\end{itemize}
\end{remark}
In the sequel will always assume the strong ellipticity of the tensor $\LL$, {\em i.e.} $\alpha_{\se}(\LL)>0$.
\par
We conclude this section with the definition of $\Gamma$-convergence of a sequence of functionals (see, {\em e.g.}, \cite{Dal,Bra}):
\begin{definition}\label{defGcon}
Let $X$ be a reflexive Banach space endowed with the metrizable weak topology on bounded sets of $X$, and let $\F^\ep:X\to\RR$ be a $\ep$-indexed sequence of functionals. The sequence $\F^\ep$ is said to $\Gamma$-converge to the functional $\F^0:X\to\RR$ for the weak topology of $X$, and we denote $\F^\ep\stackrel{\Gamma-X}{\rightharpoonup}\F^0$, if for any $u\in X$,
\begin{itemize}
\item $\forall\,u_\ep\rightharpoonup u$, $\displaystyle \F^0(u)\leq\liminf_{\ep\to 0}\F^\ep(u_\ep)$,
\item $\exists\,\bar{u}_\ep\rightharpoonup u$, $\displaystyle \F^0(u)=\lim_{\ep\to 0}\F^\ep(\bar{u}_\ep)$.
\end{itemize}
Such a sequence $\bar{u}_\ep$ is called a recovery sequence.
\end{definition}

\section{The $\Gamma$-convergence results}\label{SecGammaConvResults}
It is stated in \cite[Ch.~6, Sect.~11]{San} that the first homogenization result in linear elasticity can be found in the Duvaut work (unavailable reference). It claims that if the tensor $\LL$ is very strongly elliptic, {\em i.e.} $\alpha_{\vse}(\LL)>0$, then the solution $u^\varepsilon\in H^1_0(\Omega;\mathbb{R}^3)$ to the elasticity system \eqref{syselas} satisfies
\begin{equation}
\left\{
\begin{aligned}
& u^\varepsilon\rightharpoonup u\quad \mbox{weakly in }H^1_0(\Omega;\mathbb{R}^3),\\
&\LL^\varepsilon\nabla u^\varepsilon \rightharpoonup \LL^0\nabla u\quad  \mbox{weakly in }L^2(\Omega;\mathbb{R}^{3\times 3}),\\
& -\dive(\LL^0\nabla u ) = f,\quad
\end{aligned}
\right.
\end{equation}
where $\LL^0$ is given by
\begin{equation}\label{DefL0}
\LL^0M:M := \inf\left\{ \int_{Y_3} \LL(M+\nabla v):(M+\nabla v)\,dy,\ v\in H^1_{\per}(Y_3;\mathbb{R}^3) \right\}
\quad\mbox{for }M\in \mathbb{R}^{3\times 3},
\end{equation}
which is attained when $\Lambda_{\per}(\LL)>0$.
The previous homogenization result actually holds under the weaker assumption of functional coercivity, {\em i.e.} $\Lambda(\LL)>0$, as shown in \cite{Francfort1}.
\par
Otherwise, from the point of view of the elastic energy consider the functionals
\begin{align}
\F^\varepsilon(v) & := \int_{\Omega} \LL({x / \varepsilon})\nabla v:\nabla v\,dx, \label{DefFepsilon}\\
\F^0(v) & := \int_\Omega \LL^0\nabla v : \nabla v\,dx\label{DefF0}\quad\mbox{for }v\in H^1(\Omega,\mathbb{R}^3).
\end{align} 
Then, the following homogenization result \cite[Theorem 3.4(i)]{GMT} through the $\Gamma$-convergence of energy \eqref{DefFepsilon}, allows us to relax the very strong ellipticity of $\LL$.
\begin{theorem}[Geymonat {\em et al.} \cite{GMT}]\label{ThHomogGammaConv}
Under the conditions
\begin{equation*}
\Lambda(\LL) \ge0\quad\mbox{and}\quad\Lambda_{\per}(\LL)>0,
\end{equation*}
one has 
\begin{equation*}
\F^\varepsilon \stackrel{\Gamma-H^1_0(\Omega;\mathbb{R}^3)}\rightharpoonup \F^0,
\end{equation*}
for the weak topology of $H^1_0(\Omega;\mathbb{R}^3)$, where $\LL^0$ is given by \eqref{DefL0}.
\end{theorem}

\subsection{Generic examples of tensors satisfying $\Lambda(\LL)\geq 0$ and $\Lambda_{\rm per}(\LL)>0$}\label{SubSecApplicability}
Reference \cite{BF01} provides a class of isotropic strongly elliptic tensors for which Theorem \ref{ThHomogGammaConv} applies. However, this work is restricted to dimension two. We are going to extend the result \cite[Theorem 2.2]{BF01} to dimension three.
\par
Let us assume that there exist $p>0$ phases $Z_i,\, i = 1,\dots,p$ satisfying
\begin{equation}\label{HipZ_i}
\left\{
\begin{aligned}
& Z_i\mbox{ is open, connected and Lipschitz for any } i\in\{1,\dots,p\},	\\
& Z_i\cap Z_j = {\rm\O}\quad \forall\, i\neq j\in\{1,\dots,p\}, \\
&  \overline{Y}_3 = \bigcup_{i=1}^p \overline{Z}_i,
\end{aligned}
\right.
\end{equation}
such that the tensor $\LL$ satisfies
\begin{equation}\label{generalL}
\left\{
 \begin{array}{l}
 \LL(y) M = \lambda(y) \tr (M) I_3 + 2\mu(y)M,\quad \forall\, y\in Y_3,\forall\, M\in \mathbb{R}^{3\times 3}_s,\\
 \lambda(y)=\lambda_i,\ \mu(y)=\mu_i \mbox{ in }Z_i,\quad \forall\, i\in \{1,\dots,p\},\\
 \mu_i>0,\ 2\mu_i+\lambda_i>0,\quad \forall\, i\in \{1,\dots,p\}.
 \end{array}
\right.
\end{equation}
We further assume the existence of $d>0$ such that
\begin{equation}\label{HipExConstd}
-\min_{i=1,\dots,p}\{2\mu_i+3\lambda_i\} \le d \le 4\min_{i=1,\dots,p}\{\mu_i\}.
\end{equation}
Now, we define the following subsets of indexes
\begin{equation}\label{indexesL}
\left\{
\begin{array}{l}
I := \{i\in\{1,\dots,p\}:\ d = 4\mu_i\}, \\
J := \{ j\in\{1,\dots,p\}:\ 2\mu_j+3\lambda_j = -d \}, \\
K := \{ 1,\dots,p \} \setminus (I\cup J).
\end{array}
\right.
\end{equation}
Note that the three previous sets are disjoint. This is true, since we have $4\mu_i > -(2\mu_i+3\lambda_i)$ due to $2\mu_i+\lambda_i>0$.
\par
In this framework, we are able to prove the following theorem which is an easy extension of the two-dimensional result of \cite[Theorem 2.2]{BF01}. For the reader convenience the proof is given in the Appendix.

\begin{theorem}\label{LamdaPer>0}
Let $\LL$ be the tensor defined by \eqref{generalL} and \eqref{HipExConstd}. Then we have $\Lambda(\LL)\ge 0$.
We also have $\Lambda_{\per}(\LL) > 0$ provided that one of the two following conditions is fulfilled by the sets defined in \eqref{indexesL}:
\begin{enumerate}[{Case} 1.]
\item For each $j\in J$, there exist intervals $(a_j^-,a_j^+),(b_j^-,b_j^+)\subset [0,1]$ such that
\begin{align*}
	& (a_j^-,a_j^+)\times(b_j^-,b_j^+)\times\{0,1\}\subset \partial Z_j, \quad\mbox{or} \\
	& (a_j^-,a_j^+)\times\{0,1\}\times(b_j^-,b_j^+)\subset \partial Z_j, \quad\mbox{or} \\
	& \{0,1\}\times(a_j^-,a_j^+)\times(b_j^-,b_j^+)\subset \partial Z_j.
\end{align*}

\item For each $j\in J$, there exists $k\in K$ with $\H^2(\partial Z_j\cap \partial Z_k) >0$, where $\H^2$ denotes the 2-dimensional Hausdorff measure.
\end{enumerate}
\end{theorem}

\subsection{Relaxation of condition $\Lambda_{\rm per}(\LL)>0$}\label{SubSecRelaxationLambdaPer}

According to Theorem \ref{ThHomogGammaConv} the $\Gamma$-convergence of the functional \eqref{DefFepsilon} holds true if both $\Lambda(\LL) \ge 0$ and $\Lambda_{\per}(\LL)>0$. However, the following theorem due to Braides and the first author shows that in $N$-dimensional elasticity for $N\geq 2$, the $\Gamma$-convergence result still holds under the sole assumption $\Lambda(\LL)\ge 0$.
\begin{theorem}[Braides \& Briane] \label{BBResult} 
Let $\Omega$ be a bounded open subset of $\mathbb{R}^N$, and let $\LL$ be a bounded $Y_N$-periodic symmetric tensor-valued function in $L^\infty_{\rm per}\big(Y_N;\L_s(\RR^{N\times N})\big)$ such that
\begin{equation}\label{RealSufCond}
\Lambda(\LL)\ge 0.
\end{equation}
Then, we have
\begin{equation}\label{RealGammaConvResult}
\F^\varepsilon \stackrel{\Gamma-H^1_0(\Omega;\mathbb{R}^N)}\rightharpoonup \F^0,
\end{equation}
for the weak toplogy of $H^1_0(\Omega;\mathbb{R}^N)$, where $\F^0$ is given by \eqref{DefF0} with the tensor $\LL^0$ defined by \eqref{DefL0}.
\end{theorem}
\begin{proof}
For $\delta>0$, set $\LL_\delta:= \LL + \delta\,\mathbb{I}_s$ where $\mathbb{I}_s$ is the unit symmetric tensor, and let $\F_\delta^\varepsilon$ be the functional defined by \eqref{DefFepsilon} with $\LL_\delta$. We claim that 
 \begin{equation}\label{Lambda(L_delta)>0}
 \Lambda(\LL_\delta)> 0.
 \end{equation}
 To prove it consider $v\in C^\infty_{c}(\mathbb{R}^N;\mathbb{R}^N)$ and take $R>0$ such that $\supp v \subset B(0,R)$. Then, by \eqref{RealSufCond} we have
 \begin{equation*}
 \int_{\mathbb{R}^N} \LL_\delta\nabla v : \nabla v \,dy  = \int_{B(0,R)} \LL \nabla v : \nabla v \, dy + \delta\int_{B(0,R)} \mathbb{I}_s \nabla v : \nabla v \, dy \ge \delta \int_{B(0,R)} |e(v)|^2\, dy. 
 \end{equation*}
 By Korn's inequality there exists a constant $\alpha>0$ which {\em a priori} depends on $B(0,R)$, such that
 \begin{equation*}
\int_{B(0,R)} |e(u)|\, dy\ge\alpha \int_{B(0,R)} |\nabla v |\, dy.
 \end{equation*}
 Nevertheless, the Korn constant $\alpha$ is known to be invariant by homothetic transformations of the domain. Hence, the constant $\alpha$ actually does not depend on the radius $R$. Therefore, the two previous inequalities imply that $\Lambda(\LL_\delta)\ge \delta\alpha>0$.
 \par
Thanks to \eqref{Lambda(L_delta)>0} we can apply Theorem \ref{ThHomogGammaConv} with the functional $\F_\delta^\varepsilon$.
Hence, $\F_\delta^\varepsilon \stackrel{\Gamma}\rightharpoonup \F_\delta^0$ for the weak topology of $H^1_0(\Omega;\mathbb{R}^N)$, where 
\[
\F_\delta^0(u) := \int_\Omega \LL_\delta^0 \nabla u : \nabla u \, dx\quad \mbox{for } u\in H^1_0(\Omega,\mathbb{R}^N),
\]
and $\LL_\delta^0$ is given by \eqref{DefL0} with $\LL=\LL_\delta$.
\newline\indent
On the one hand, since $H^1_0(\Omega;\mathbb{R}^N)$ is a separable metric space, up to subsequence there exists the $\Gamma$-limit of $\F^\varepsilon$ for the weak topology of $H^1_0(\Omega;\mathbb{R}^N)$ as $\varepsilon\to0$. Fix $u\in H^1_0(\Omega;\mathbb{R}^N)$, and consider a recovery sequence $u_\varepsilon$ for $\F^\varepsilon$ (see Definition~\ref{defGcon}) which converges weakly to $u$ in $H^1_0(\Omega;\mathbb{R}^N)$.
Since $u_\varepsilon$ is bounded in $H^1_0(\Omega,\mathbb{R}^N)$, we have
\begin{align*}
(\Gamma\mbox{-}\lim \F^\varepsilon)(u) 	&\le \F_\delta^0 (u) \\
										&\le \liminf_{\varepsilon\to0} \int_\Omega \LL_\delta(x/\varepsilon)\nabla u_\varepsilon : \nabla u_\varepsilon\, dx\\
										&\le \liminf_{\varepsilon\to0} \int_\Omega \LL(x/\ep) \nabla u_\varepsilon : \nabla u_\varepsilon\, dx + {O}(\delta)\\
										&= (\Gamma\mbox{-}\lim \F^\varepsilon) (u) + {O}(\delta),						     	 
\end{align*}
which implies that $\F_\delta^0(u)$ converges to $\F^0(u)$ as $\delta\to0$.
\par
On the other hand, let $\LL^0$ be given by \eqref{DefL0}. For $\eta>0$ and for $M\in \mathbb{R}^{N\times N}$, consider a function $\varphi_\eta$ in $H^1_{\per}(Y_N;\mathbb{R}^N)$ such that
$$
\int_{Y_N} \LL(y)(M+\nabla \varphi_\eta): (M+\nabla \varphi_\eta)\, dy\leq \LL^0 M:M + \eta.
$$
We then have
\begin{align*}
\LL^0 M:M 	& \le \LL_\delta^0 M:M \\
				&\le \int_{Y_N} \LL_\delta(y)(M + \nabla \varphi_\eta):(M + \nabla \varphi_\eta)\, dy\\
				& \le \int_{Y_N}\LL(y)(M + \nabla\varphi_\eta):(M + \nabla\varphi_\eta)\, dy + O_\eta(\delta).
\end{align*}
Hence, making $\delta$ tend to 0 for a fixed $\eta$, we obtain
\begin{align*}
\LL^0 M:M 	& \le \liminf_{\delta\to0}\  (\LL_\delta^0 M : M) \\
				& \le \limsup_{\delta\to0}\ (\LL_\delta^0 M : M) \\
				& \le \int_{Y_N} \LL(y) (M + \nabla\varphi_\eta) : (M + \nabla\varphi_\eta) \, dy \\
				& \le \LL^0 M : M + \eta.
\end{align*}
Due to the arbitrariness of $\eta$, we get that $\LL^0_\delta$ converges to $\LL^0$ as $\delta\to0$.\newline\indent
Therefore, by the Lebesgue dominated convergence theorem we conclude that for any $u\in H^1_0(\Omega;\mathbb{R}^N)$,
\[
\F^0(u)= \lim_{\delta\to0} \F_\delta^0 (u)=\lim_{\delta\to0} \int_\Omega \LL^0_\delta \nabla u : \nabla u\, dx= \int_\Omega \LL^0\nabla u : \nabla u\, dx.
\]
\end{proof}

\section{Loss of ellipticity in three-dimensional linear elasticity through the homogenization of a laminate}\label{LossOfEllipticity}
In this section we will construct an example of a three-dimensional strong elliptic material $\LL$ which is weakly coercive, {\em i.e.} $\Lambda(\LL)\ge 0$, but for which the strong ellipticity is lost through homogenization. Firstly, let us recall the following result due to Guti\'{e}rrez \cite{G01}.
\begin{proposition}[Guti\'{e}rrez \cite{G01}]\label{GutProp}
For any strongly, but not semi-very strongly elliptic isotropic material, referred to as material $a$, there are very strongly elliptic isotropic materials such that if we laminate them with material $a$, in appropriately chosen proportions and directions, we generate an effective elasticity tensor that is not strongly elliptic.
\end{proposition}
\begin{remark}[Isotropic tensors] The elasticity tensor $\LL\in L^\infty\big(Y_3;\L_s(\mathbb{R}^{3\times 3})\big)$ of an isotropic material is given by
\begin{equation*}
\LL(y)M=\lambda(y)\tr(M) I_3  + 2\mu(y)M,\quad\mbox{for }y\in Y_3\mbox{ and }M\in \mathbb{R}^{3\times 3}_s,
\end{equation*}
where $\lambda$ and $\mu$ are the Lam\'e coefficients of $\LL$.\newline
As a consequence, we have 
\begin{equation*}
\alpha_{\se}(\LL) = \essinf_{y\in Y_3}\big( \min\{ \mu(y),2\mu(y) + \lambda(y)  \} \big),
\end{equation*}
\begin{equation*}
\alpha_{\vse}(\LL) = \essinf_{y\in Y_3}\big( \min\{ \mu(y),2\mu(y) + 3\lambda(y) \} \big). 
\end{equation*}
\end{remark}
\par
Here is a summary of the proof of Proposition \ref{GutProp}.
Consider two isotropic, homogeneous tensors $\LL_a$ and $\LL_b$ such that $\LL_a$ is strongly elliptic, {\em i.e.}
\begin{equation*}
\lambda_a+2\mu_a>0,\quad \mu_a>0,
\end{equation*}
but not semi-very strongly elliptic, {\em i.e.}
\begin{equation*}
3\lambda_a+2\mu_a < 0.
\end{equation*}
and such that $\LL_b$ is very strongly elliptic, {\em i.e.}
\begin{equation*}
3\lambda_b+2\mu_b>0,\quad\mu_b>0.
\end{equation*}
Considering the rank-one laminate in the direction $y_1$ mixing $\LL_a$ with volume fraction $\theta_1\in (0,1)$ and $\LL_b$ with volume fraction $(1-\theta_1)$, Guti\'errez \cite{G01} proved that the effective tensor $\LL_1^*$ in the sense of Murat-Tartar $1^*$-convergence (see, {\em e.g.}, \cite[Section~3]{G01}) satisfies the following properties:
\begin{itemize}
\item If $0\le\mu_a+\lambda_a$, then
\begin{equation*}
\alpha_{\se}(\LL^*_1)>0.
\end{equation*}
\item If $-\mu_b\le\mu_a+\lambda_a<0$, then
\begin{equation*}
\alpha_{\se}(\LL^*_1)
\begin{cases}
= 0&\mbox{if } \mu_b = -\mu_a-\lambda_a,\\
\ge 0&\mbox{if }-\mu_a-\lambda_a<\mu_b\le -{1\over4}(2\mu_a+3\lambda_a),\\
> 0&\mbox{if }-{1\over 4}(2\mu_a+3\lambda_a)<\mu_b.
\end{cases}
\end{equation*}
\item The case $\mu_a + \lambda_a < -\mu_b$ is disposed of, since $\LL_1^*$ does not even satisfy the Legendre-Hadamard condition.
\end{itemize}
In the case where $\alpha_{\rm se}(\LL^*_1)>0$,  Guti\'{e}rrez (see \cite[Section 5.2]{G01}) performed a second lamination in the direction $y_2$ mixing the anisotropic material generated by the first lamination with volume fraction $\theta_2\in(0,1)$, and a suitable very strongly elliptic isotropic material ($\LL_c,\mu_c,\lambda_c$) with volume fraction $(1-\theta_2)$. In this way he derived a rank-two laminate of effective tensor $\LL^*_2$ which is not strongly elliptic.
\par
In this section we will try to find a general class of periodic laminates for which the strong ellipticity is lost through homogenization. To this end we will extend to dimension three the rank-one lamination approach of \cite{BF01} performed in dimension two. However, the outcome is surprisingly different from that of the two-dimensional case of \cite{BF01}. Indeed, we will prove in the first subsection that it is not possible to lose strong ellipticity by a rank-one lamination through homogenization following the two-dimensional approach of \cite{BF01}. This is the reason why we will perform a second lamination in the second part of the section.

\subsection{Rank-one lamination}\label{SubSecRank1Lam}
In this subsection we are going to focus on the rank-one lamination. As noted before, in the two-dimensional case of \cite{BF01} it was proved a necessary and sufficient condition for a general rank-one laminate to lose strong ellipticity. Mimicking the same approach in dimension three we obtain the following quite different result.

\begin{theorem}\label{ThNoLossFirstLaminate}
Let $\LL\in L^\infty_{\per}\big(Y_1;\L_s(\mathbb{R}^{3\times 3})\big)$ be a $Y_1$-periodic isotropic tensor-valued function which is strongly elliptic, {\em i.e.} $\alpha_{\se}(\LL)>0$. Assume that $\Lambda_{\per}(\LL)>0$ and that there exists a constant matrix $D\in\mathbb{R}^{3\times3}$ such that
\begin{equation}\label{CondLM:M+D:Cof(M)>0}
\LL(y_1)M:M + D:\Cof(M)\geq0,\quad \mbox{a.e. }y_1\in Y_1,\quad\forall\, M\in\mathbb{R}^{3\times 3}.
\end{equation}
Then, the homogenized tensor $\LL^0$ defined by \eqref{DefL0} is strongly elliptic, {\em i.e.} $\alpha_{\se}(\LL^0)>0$.
\end{theorem}
\begin{remark}
In dimension two for any periodic function $\varphi\in H^1_{\per}(Y_2;\mathbb{R}^2)$, the only null lagrangian (up to a multiplicative constant) is the determinant of $\nabla\varphi$. Although the two-dimensional case seems {\em a priori} more restrictive than the three-dimensional case from an algebraic point of view, the two-dimensional Theorem~3.1 of \cite{BF01} shows that for a suitable isotropic tensor $\LL=\LL(y_1)$, satisfying for some constant $d\in\mathbb{R}$, the condition
\begin{equation}\label{Transine2D}
\LL(y_1)M:M+d \det(M)\ge0,\quad \mbox{a.e. in }Y_1,\ \forall\, M\in\mathbb{R}^{2\times2},
\end{equation}
it is possible to lose strong ellipticity through homogenization. On the contrary, the three-dimensional Theorem~\ref{ThNoLossFirstLaminate} shows that it is not possible to lose strong ellipticity under condition \eqref{CondLM:M+D:Cof(M)>0} which is the natural three-dimensional extension of \eqref{Transine2D}.
\end{remark}
\begin{remark}\label{RemarkPartialResult}
Observe that condition \eqref{CondLM:M+D:Cof(M)>0} implies that $\LL$ is weakly coercive, {\em i.e.} $\Lambda(\LL)\ge0$, but the converse is not true in general. Therefore, it might be possible to find a weakly coercive, strongly elliptic isotropic tensor $\LL=\LL(y_1)$ for which the strong ellipticity is lost. However, we have not succeeded in deriving such a tensor.
\end{remark}
\begin{remark}
In the proof of Proposition~\ref{GutProp} Guti\'errez implicitly proved the result of Theorem \ref{ThNoLossFirstLaminate} when the matrix $D$ has the form $D=dI_3$ and $\LL$ is of the type
\begin{equation*}
\LL(y_1) = \chi(y_1)\,\LL_a +\big(1-\chi(y_1)\big)\,\LL_b.
\end{equation*}
Moreover, it is worth mentioning that the cases for which Guiti\'errez obtained the loss of ellipticity with a rank-one lamination do not contradict Theorem \ref{ThNoLossFirstLaminate}, since in those cases condition \eqref{CondLM:M+D:Cof(M)>0} does not hold.
\end{remark}
The rest of this subsection is devoted to the proof of Theorem \ref{ThNoLossFirstLaminate}.
For any $Y_1$-periodic tensor-valued function $\LL\in L^\infty_{\per}\big(Y_1;\L_s(\mathbb{R}^{3\times 3})\big)$ which is strongly elliptic, {\em i.e.} $\alpha_{\se}(\LL)>0$, define for a.e. $y_1\in Y_1$, the $y_1$-dependent inner product
\begin{equation*}
(\xi,\eta)\in \mathbb{R}^3\times \mathbb{R}^3\mapsto \LL(y_1)(\xi\otimes e_1):(\eta\otimes e_1).
\end{equation*}
It is indeed an inner product because $\alpha_{\se}(\LL)> 0$. The matrix-valued function 
\begin{multline}\label{DefMatrixL}
L(y_1)= \begin{pmatrix}
 l_1(y_1) & l_{12}(y_1) & l_{13}(y_1)\cr
 l_{12}(y_1)  & l_2(y_1) & l_{23}(y_1) \cr
 l_{13}(y_1) & l_{23}(y_1) & l_3(y_1) 
\end{pmatrix} :=\\
\begin{pmatrix}
 \LL(y_1)(e_1 \otimes e_1):(e_1 \otimes e_1) & \LL(y_1) (e_1 \otimes e_1):(e_2 \otimes e_1)& \LL(y_1)(e_1 \otimes e_1):(e_3 \otimes e_1)\cr
 \LL(y_1)(e_1 \otimes e_1):(e_2 \otimes e_1)  & \LL(y_1)(e_2 \otimes e_1):(e_2 \otimes e_1) & \LL(y_1)(e_2 \otimes e_1):(e_3 \otimes e_1) \cr
 \LL(y_1)(e_1 \otimes e_1):(e_3 \otimes e_1) & \LL(y_1)(e_2 \otimes e_1):(e_3 \otimes e_1) & \LL(y_1) (e_3 \otimes e_1):(e_3 \otimes e_1)
\end{pmatrix}
\end{multline}
is therefore symmetric positive definite.
\par
Similarly to \cite[Lemma 3.3]{BF01} the next result provides an estimate which is a direct consequence of condition~\eqref{CondLM:M+D:Cof(M)>0} with a matrix of the type $D=dI_3$. Observe that for the moment we are not assuming that the tensor $\LL$ is isotropic.
\begin{lemma}\label{LemmaAnisotropic}
Let $\LL\in L^\infty_{\per}(Y_1;\L_s(\mathbb{R}^{3\times 3}))$ be a $Y_1$-periodic bounded tensor-valued function with $\Lambda_{\per}(\LL)>0$. Assume the existence of a constant $d\in\mathbb{R}$ such that $\LL$ satisfies condition \eqref{CondLM:M+D:Cof(M)>0} with $D=dI_3$. Then, we have
\begin{equation}\label{L>Q}
\LL(y_1)M:M \ge Q(M), \quad \mbox{a.e. in } Y_1, \ \forall\, M\in \mathbb{R}^{3\times3},\ \mbox{$M$ rank-one,}
\end{equation}
where
\begin{align*}
& Q(M) :=\\
&{\det(\tilde{L}^{11}) \over \det(L)}\left( \LL M:(e_1\otimes e_1) + {d\over 2}{M_{33}} + {d\over 2}{M_{22}} \right)^2  +  {\det(\tilde{L}^{22}) \over \det(L)}  \left( \LL M:(e_2\otimes e_1) - {d\over 2}{M_{12}} \right)^2 \\
&+ {\det(\tilde{L}^{33})\over \det(L)}\left( \LL M:(e_3\otimes e_1) - {d\over 2}{M_{13}} \right)^2  - {2\det(\tilde{L}^{12}) \over \det(L)} \left( \LL M:(e_1\otimes e_1) + {d\over 2}{M_{33}} + {d\over 2}{M_{22}} \right)\left( \LL M:(e_2\otimes e_1) - {d\over 2}{M_{12}} \right) \\
& + {2 \det(\tilde{L}^{13})\over \det(L)}\left( \LL M:(e_1\otimes e_1) + {d\over 2}{M_{33}} + {d\over 2}{M_{22}} \right)\left( \LL M:(e_3\otimes e_1) - {d\over 2}{M_{13}} \right)\\
& - {2\det(\tilde{L}^{23})\over \det(L)}\left( \LL M:(e_2\otimes e_1) - {d\over 2}{M_{12}} \right)\left( \LL M:(e_3\otimes e_1) - {d\over 2}{M_{13}} \right).
\end{align*}

Furthermore, if $\LL^0$ is the homogenized tensor of $\LL$, then $\alpha_{\se}(\LL^0) = 0$ if and only if there exists a rank-one matrix $M$ such that 
\begin{equation}\label{LM:M=Q(M)}
\LL(y_1)M:M = Q(M),\quad\mbox{a.e. in }Y_1,
\end{equation}
together with
\begin{equation}\label{integrEqs}
\left\{
\begin{aligned}
	&\int_{Y_1} {\det(\tilde{L}^{13})\over \det(L)}(t) \left( \LL(t)M:(e_1\otimes e_1) + {d\over 2}{M_{22}} + {d\over 2}{M_{33}} \right)dt \\
	&= \int_{Y_1} \left[ {\det(\tilde{L}^{23})\over\det(L)}(t)\left(\LL(t)M:(e_2\otimes e_1) - {d\over 2}{M_{12}}\right) - { \det(\tilde{L}^{33})\over \det(L)}(t)\left(\LL(t)M:(e_3\otimes e_1) - {d\over 2}{M_{13}}\right) \right]dt, \\ \\
	&\int_{Y_1}  {\det(\tilde{L}^{12})\over \det(L)}(t)\left( \LL(t)M:(e_1\otimes e_1) + {d\over 2}{M_{22}} + {d\over 2}{M_{33}} \right)dt \\
	&= \int_{Y_1} \left[ {\det(\tilde{L}^{22})\over \det(L)}(t) \left( \LL(t)M:(e_2\otimes e_1) - {d\over 2}{M_{12}}\right) - {\det(\tilde{L}^{23})\over \det(L)}(t)\left( \LL(t)M:(e_3\otimes e_1) - {d\over 2}{M_{13}}\right)   \right]dt, \\ \\
	&\int_{Y_1}{\det(\tilde{L}^{11})\over\det(L)} (t)\left( \LL(t)M:(e_1\otimes e_1) + {d\over 2}{M_{22}} + {d\over 2}{M_{33}} \right)dt \\
	&= \int_{Y_1} \left[	{ \det(\tilde{L}^{12})\over \det(L)}(t)\left( \LL(t)M:(e_2\otimes e_1) - {d\over 2}{M_{12}} \right) - 
	{\det(\tilde{L}^{13})\over \det(L)} (t) \left( \LL(t)M:(e_3\otimes e_1) - {d\over 2}{M_{13}} \right)\right] dt.
\end{aligned}
\right.
\end{equation}
\end{lemma}
Finally, we state a corollary of the previous result in the particular case of isotropic tensors. 
\begin{lemma}\label{LemNoLossStrEllipIsotrop}
Let $\LL\in L^\infty_{\per}(Y_1;\L_s(\mathbb{R}^{3\times 3}))$ be a $Y_1$-periodic bounded isotropic tensor-valued function with $\Lambda_{\per}(\LL)>0$. Assume that there exists a constant $d\in \mathbb{R}$ such that the Lam\'e coefficients of $\LL(y_1)$ satisfy
 \begin{equation}\label{condLemmaIsotr}
\max\{0,-2\mu(y_1)-3\lambda(y_1)\} \le d \le 4\mu(y_1) \quad\mbox{for a.e. $y_1$ in }Y_1.
\end{equation}
Then, the homogenized tensor $\LL^0$ defined by \eqref{DefL0} is strongly elliptic.
\end{lemma}
Thanks to the previous lemmas, we are now able to demonstrate the main result of this section.
\begin{proof}[Proof of Theorem \ref{ThNoLossFirstLaminate}]
Firstly, assume that \eqref{CondLM:M+D:Cof(M)>0} is satisfied with the matrix $D$ being of the type $D=dI_3$ for some $d\in \mathbb{R}$. This is equivalent to condition \eqref{condLemmaIsotr}, as it was proved by Guti\'errez in \cite[Section 4.2]{G01}. By virtue of Lemma \ref{LemNoLossStrEllipIsotrop}, $\LL^0$ is strongly elliptic, which concludes the proof in this case.
\par
In the sequel we will show that if there exists a constant matrix $D\in\mathbb{R}^{3\times3}$ such that condition \eqref{CondLM:M+D:Cof(M)>0} is fulfilled, then there exists a constant $d\in \mathbb{R}$ such that \eqref{CondLM:M+D:Cof(M)>0} holds with $D=dI_3$. This combined with Lemma~\ref{LemNoLossStrEllipIsotrop} implies that $\LL^0$ is strongly elliptic.
\par
Assume that \eqref{CondLM:M+D:Cof(M)>0} holds for some matrix $D\in\mathbb{R}^{3\times3}$, namely for any $M\in\mathbb{R}^{3\times3}$,  we have a.e. in $Y_1$,
\begin{align*}
0\le &\ \lambda(M_{11}+M_{22}+M_{33})^2 \\
& + 2\mu\left(M_{11}^2 + M_{22}^2 + M_{33}^2 + 2\left[ \left( M_{12} + M_{21} \over 2 \right)^2 + \left( M_{13} + M_{31} \over 2 \right)^2 + \left( M_{23} + M_{32} \over 2 \right)^2 \right] \right) \\
&+ D_{11}(M_{22}M_{33} - M_{23}M_{32}) - D_{12}(M_{21}M_{33} - M_{23}M_{31}) + D_{13}(M_{21}M_{32} - M_{22}M_{31}) \\
& - D_{21}(M_{12}M_{33} - M_{13}M_{32}) + D_{22}(M_{11}M_{33} - M_{13}M_{31}) - D_{23}(M_{11}M_{32} - M_{12}M_{31}) \\
& + D_{31}(M_{12}M_{23} - M_{13}M_{22}) - D_{32}(M_{11}M_{23} - M_{13}M_{21}) + D_{33}(M_{11}M_{22} - M_{12}M_{21}).
\end{align*}
The previous condition is equivalent to the following matrix being positive semi-definite a.e. in $Y_1$
\begin{equation*}
\begin{pmatrix}
\lambda + 2\mu & \lambda + {D_{33}\over2} & \lambda + {D_{22}\over2} &0&0&0&0& -{D_{32}\over2} & -{D_{23}\over2} \\ \\
\lambda + {D_{33}\over2} & \lambda+2\mu & \lambda + {D_{11}\over2} &0&0&-{D_{31}\over 2} & {D_{13}\over 2} & 0 & 0 \\ \\
\lambda + {D_{22}\over 2} & \lambda + {D_{11}\over 2} & \lambda + 2\mu & -{D_{21}\over 2} & -{D_{12}\over 2} &0 &0&0&0 \\ \\
0 & 0 & -{D_{21}\over 2} & \mu & \mu - {D_{33}\over 2} & 0 &  {D_{23}\over 2} & {D_{31}\over 2} & 0 \\  \\
 0 & 0 & -{D_{12}\over 2} & \mu - {D_{33}\over 2} & \mu & {D_{32}\over 2} & 0 & 0 & {D_{13}\over 2} \\ \\
 0 & -{D_{31}\over 2} & 0 & 0 & {D_{32}\over 2} & \mu & \mu - {D_{22}\over 2} & 0 & {D_{21}\over 2} \\ \\
 0 & -{D_{13}\over 2} & 0 & {D_{23}\over 2} & 0 & \mu - {D_{22}\over 2} & \mu & {D_{12}\over 2} & 0 \\ \\
 -{D_{32}\over 2} & 0 & 0 & {D_{31}\over 2} & 0 & 0 & {D_{12}\over 2} & \mu & \mu - {D_{11}\over 2} \\ \\
 -{D_{23}\over 2} & 0 & 0 & 0 &{D_{13}\over 2} & {D_{21}\over 2} & 0 & \mu - {D_{11}\over 2} & \mu
\end{pmatrix}.
\end{equation*}
In particular, this implies that the following matrices are positive semi-definite a.e. in $Y_1$:
\begin{equation}\label{2x2Matrix}
\begin{pmatrix}
\mu & \mu - {D_{ii}\over2} \\ \\
\mu-{D{ii}\over2} & \mu
\end{pmatrix} \quad\mbox{for }i=1,2,3,
\end{equation}
\begin{equation}\label{3x3Matrix}
B:=
\begin{pmatrix}
\lambda + 2\mu & \lambda + {D_{33}\over2} & \lambda + {D_{22}\over2} \\ \\
\lambda + {D_{33}\over 2} & \lambda + 2\mu & \lambda + {D_{11}\over 2} \\ \\
\lambda + {D_{22}\over 2} & \lambda + {D_{11}\over 2} & \lambda + 2\mu
\end{pmatrix}.
\end{equation}
\par
Now, we will prove that there exists $i\in \{1,2,3\}$ such that
\begin{equation}\label{CondContradiction}
-\essinf_{y_1\in Y_1}\{ 2\mu(y_1) + 3\lambda(y_1) \}\le D_{ii} \le 4\essinf_{y_1 \in Y_1} \{ \mu(y_1) \}.
\end{equation}
Note that we can assume 
\begin{equation}\label{LNotVSE}
 \essinf_{y_1\in Y_1} \{2\mu(y_1) + 3\lambda(y_1)\} < 0.
\end{equation}
Otherwise, since the matrix \eqref{2x2Matrix} is positive semi-definite, or equivalently
\begin{equation}\label{0Dii4mu}
0\leq D_{ii} \leq 4\essinf_{y_1\in Y_1}\{\mu({y_1})\} \quad\mbox{for }i=1,2,3,
\end{equation} 
condition \eqref{CondContradiction} holds immediately.
\par
We assume by contradiction that \eqref{CondContradiction} is violated for any $i=1,2,3$. Since the matrix $B$ defined by \eqref{3x3Matrix} is positive semi-definite, we get for any $i=1,2,3$,
\begin{equation*}
\begin{vmatrix}
\lambda + 2\mu & \lambda + {D_{ii}\over 2} \\ \\
\lambda + {D_{ii}\over2} & \lambda + 2\mu
\end{vmatrix}
\ge 0 \quad \mbox{a.e. in }Y_1,
\end{equation*}
which is equivalent to 
\begin{equation*}
-\,4\essinf_{y_1\in Y_1}  \{ \mu(y_1) + \lambda(y_1) \}\le D_{ii} \le 4 \essinf_{y_1\in Y_1} \{ \mu(y_1) \} \quad\mbox{for }i=1,2,3.
\end{equation*}
Since by assumption \eqref{CondContradiction} is not satisfied for any $i=1,2,3$ and \eqref{0Dii4mu} holds, then the previous condition yields
\begin{equation}\label{BoundsOnDii}
-4\essinf_{y_1\in Y_1} \{ \mu(y_1)  + \lambda(y_1)\} \le D_{ii} < -\essinf_{y_1\in Y_1}\{2\mu(y_1) + 3\lambda(y_1)\}\quad\mbox{for }i=1,2,3.
\end{equation}
Set $d:=\max_{i=1,2,3}\{D_{ii}\}$. By \eqref{BoundsOnDii} there exists $\varepsilon>0$ such that
\begin{equation}\label{CondEpsilon}
d + \varepsilon < -\essinf_{y_1\in Y_1}\{2\mu(y_1) + 3\lambda(y_1)\}.
\end{equation}
Define the set $P_\varepsilon\subset Y_1$ by
\begin{equation*}
P_\varepsilon:=\left\{x_1\in Y_1: 2\mu(x_1) + 3\lambda(x_1) < \essinf_{y_1\in Y_1} \{2\mu(y_1) + 3\lambda(y_1)\} + \varepsilon   \right\}.
\end{equation*}
It is clear that $|P_\varepsilon|>0$, and from \eqref{CondEpsilon} and the definition of $P_\varepsilon$ we obtain
\begin{equation*}
d + \varepsilon  < -\essinf_{y_1\in Y_1} \{2\mu(y_1) + 3\lambda(y_1)\} < - \big(2\mu(x_1) + 3\lambda(x_1)\big) + \varepsilon\quad \mbox{a.e. } x_1\in P_\varepsilon,
\end{equation*}
which leads to
\begin{equation}\label{Lambda+Dii/2<0}
\lambda(x_1) + {d\over2} < -{1\over2}\big(\lambda(x_1)+2\mu(x_1)\big)<0\quad \mbox{a.e. }x_1\in P_\varepsilon.
\end{equation} 
Since the matrix $B$ from \eqref{3x3Matrix} is positive semi-definite, then, its determinant is non-negative a.e. in $Y_1$. In particular we have
\begin{equation}\label{DetNoNegative}
\begin{aligned}
0\leq \det\big(B(x_1)\big)=&\ \big(\lambda(x_1) + 2\mu(x_1)\big)^3 + 2\left(\lambda(x_1) + {D_{11}\over2}\right)\left(\lambda(x_1) + {D_{22}\over2}\right)\left(\lambda(x_1) + {D_{33}\over2}\right) \\
	& - \big(\lambda(x_1) + 2\mu(x_1)\big)\left[ \left(\lambda(x_1) + {D_{11}\over2}\right)^2 + \left(\lambda(x_1) + {D_{22}\over2}\right)^2 + \left(\lambda(x_1) + {D_{33}\over2}\right)^2 \right],
\end{aligned}
\end{equation}
a.e. $x_1\in P_\varepsilon$. Then, it follows that
\begin{equation}\label{|B|<=number}
\det\big(B(x_1)\big)\le \big(\lambda(x_1) + 2\mu(x_1)\big)^3 + 2\left( \lambda(x_1) +  {d\over2}\right)^3 - 3\big(\lambda(x_1) + 2\mu(x_1)\big)\left( \lambda(x_1) + {d\over2} \right)^2\quad \mbox{a.e. }x_1\in P_\varepsilon.
\end{equation}
To derive a contradiction let us show that the right-hand side of inequality \eqref{|B|<=number} is negative.
By \eqref{Lambda+Dii/2<0} we get
\begin{equation*}
4\left( \lambda(x_1) + {d\over2} \right)^2 > \big(\lambda(x_1) + 2\mu(x_1)\big)^2 \quad\mbox{a.e. }x_1\in P_\varepsilon,
\end{equation*}
which, multiplying by $\lambda(x_1) + 2\mu(x_1) > 0$, leads to
\begin{equation*}
\big(\lambda(x_1) + 2\mu(x_1)\big)^3 - 4\big(\lambda(x_1) + 2\mu(x_1)\big)\left( \lambda(x_1) + {d\over2} \right)^2 <0\quad\mbox{a.e. }x_1\in P_\varepsilon.
\end{equation*}
Again using $\eqref{Lambda+Dii/2<0}$ we deduce that 
\begin{equation*}
2\left(\lambda(x_1) + {d\over2}\right)^3 < -\big(\lambda(x_1) + 2\mu(x_1)\big)\left( \lambda(x_1) + {d\over2} \right)^2\quad\mbox{a.e. }x_1\in P_\varepsilon.
\end{equation*}
Adding the two last inequalities we obtain
\begin{equation*}
\big(\lambda(x_1) + 2\mu(x_1)\big)^3 + 2\left( \lambda(x_1) +  {d\over2}\right)^3 - 3\big(\lambda(x_1) + 2\mu(x_1)\big)\left( \lambda(x_1) + {d\over2} \right)^2 < 0\quad\mbox{a.e. }x_1\in P_\varepsilon,
\end{equation*}
which by \eqref{|B|<=number} implies that $\det(B)<0$ in $P_\varepsilon$, a contradiction with \eqref{DetNoNegative}.
\par
Therefore, condition \eqref{CondContradiction} is satisfied by $D_{ii}\geq 0$ (due to \eqref{0Dii4mu}) for some $i=1,2,3$. Hence, condition \eqref{condLemmaIsotr} holds with $d=D_{ii}$, or equivalently \eqref{CondLM:M+D:Cof(M)>0} is satisfied by the matrix $D_{ii}I_3$, which concludes the proof.
\end{proof}
Now, let us prove the auxiliary results of the section.
\begin{proof}[Proof of Lemma \ref{LemmaAnisotropic}]
Let $M\in \mathbb{R}^{3\times 3}$ be a rank-one matrix. Then, we have $\det(M)= 0$, and $\adj_{ii}(M) = 0$ for $i=1,2,3$. Therefore, we get
\begin{equation}\label{L0MM>=0}
\begin{aligned}
\LL^0M:M &=
 	\min\left\{ \int_{Y_3} \LL(M+\nabla\varphi):(M+\nabla\varphi)\,dy: \varphi \in H^1_{\per}(Y_3;\mathbb{R}^3)\right\} \\
	& = \min\left\{ \int_{Y_3} \big( \LL(M+\nabla\varphi)(M+\nabla\varphi) + dI_3:\Cof(M+\nabla\varphi): \varphi \in H^1_{\per}(Y_3;\mathbb{R}^3)			 \big)\, dy \right\} \ge 0.
\end{aligned}
\end{equation}
Take $\varphi= \varphi(y_1)=(\varphi_1,\varphi_2,\varphi_3)\in C^1_{\per}(Y_1;\mathbb{R}^3)$. Then, the matrix
\begin{equation*}
\nabla\varphi = \varphi'\otimes e_1 = \varphi'_1(e_1\otimes e_1) + \varphi'_2(e_2\otimes e_1) + \varphi'_3(e_3\otimes e_1),
\end{equation*}
is a rank-one (or the null) matrix. Also, note that
\[\adj_{ij}(M) = (-1)^{i+j}\det(\tilde{M}^{ji}).\]
Considering the previous expressions, from \eqref{CondLM:M+D:Cof(M)>0} it follows that
\begin{align*}
	0\le &\ \LL (M+\nabla\varphi):(M+\nabla\varphi)+ d \sum_{i=1}^3 \adj_{ii}(M+\nabla\varphi) \\
=&\	\LL M:M + 2\LL M:(e_1\otimes e_1)\varphi'_1+2\LL M:(e_2\otimes e_1)\varphi'_2 + 2\LL M:(e_3\otimes e_1)\varphi'_3 + l_1(\varphi'_1)^2 + 2l_{12}\varphi'_1\varphi'_2   \\
 &+ 2l_{13}\varphi'_1\varphi'_3+ l_2(\varphi'_2)^2 + 2l_{23}\varphi'_2\varphi'_3 + l_2(\varphi'_3)^2+ d(M_{33}\varphi'_1 - M_{13}\varphi'_3 + M_{22}\varphi'_1 - M_{12}\varphi'_2) \\
 = & \  \LL M:M + l_1(\varphi'_1)^2 + l_2(\varphi'_2)^2 + l_3(\varphi'_3) + 2l_{12}\varphi'_1\varphi'_2 + 2l_{13}\varphi'_1\varphi'_3 + 2l_{23}\varphi'_2\varphi'_3 \\
 &\ \big[2\LL M:(e_1\otimes e_2) + d(M_{33}+dM_{22})\big]\varphi'_1 + \big[2\LL M:(e_2\otimes e_1) - dM_{12}\big]\varphi'_2 + \big[2\LL M:(e_3\otimes e_1) - dM_{13}\big]\varphi'_3.
\end{align*}
For the previous equalities we have used that
\[
\adj_{ii}(A+B) = \adj_{ii}(A) +\adj_{ii}(B) + \Cof (\tilde A^{ii}):\tilde B^{ii}.
\]
The purpose is to rewrite the last expression as the sum of squares. With that in mind, one obtains
\begin{equation}\label{0<L-Q+more}
\begin{aligned}
	0\le &\ \LL (M+\nabla\varphi):(M+\nabla\varphi)+ dI_3:\Cof(M+\nabla\varphi) \\
	= & \ \LL M:M - Q(M) + l_1\left[ \varphi'_1 + {l_{12} \over l_1}\varphi'_2 + {l_{13}\over l_1}\varphi'_3 + {1\over l_1} \left(\LL M:(e_1\otimes e_1) + {d\over 2}{M_{22}} + {d\over 2}{M_{33}}\right) \right]^2 \\
	& + {\det(\tilde{L}^{33}) \over  l_1}\left[ \varphi'_2 + {\det(\tilde{L}^{23}) \over \det(\tilde{L}^{33})} \varphi'_3 - {l_{12} \over \det(\tilde{L}^{33})}\left(\LL M:(e_1\otimes e_1) + {d\over 2}{M_{22}} + {d\over 2}{M_{33}}\right) \right. \\
	& \left. + {l_1\over \det(\tilde{L}^{33})} \left(\LL M:(e_2\otimes e_1) - {d\over 2}{M_{12}}\right) \right]^2 \\
	& + {\det(L)\over \det(\tilde{L}^{33})} \left[ \varphi'_3 + {\det(\tilde{L}^{13}) \over \det(L)}\left(\LL M:(e_1\otimes e_1) + {d\over 2}{M_{22}} + {d\over 2}{M_{33}} \right)  - {\det(\tilde{L}^{23})\over\det(L)}\left( \LL M:(e_2\otimes e_1) - {d\over 2}{M_{12}} \right)\right. \\
	&\left.+ {\det(\tilde{L}^{33})\over \det(L)}\left(\LL M:(e_3\otimes e_1) - {d\over 2}{M_{13}}\right)\right]^2.
\end{aligned}
\end{equation}
Since $\varphi'_1,\varphi'_2$ and $\varphi'_3$ can be chosen arbitrarily, the three square brackets in the previous equality can be equated to 0 at any Lebesgue point $y_1\in Y_1$ of $\LL$, and thus \eqref{L>Q} holds. Using a density argument the previous equality also holds a.e. in $Y_1$, for any $\varphi\in H^1_{\per}(Y_1;\mathbb{R}^3)$.
\par
Now, we are going to prove the second part of Lemma~\ref{LemmaAnisotropic}. Assume $\LL^0$ is not strongly elliptic. Then, there exists a rank-one matrix $M$ such that $\LL^0M:M = 0$. Taking into account expressions \eqref{L0MM>=0} the minimizer $v_M$ associated with $\LL^0M:M$ (see \cite[Lemma~3.2]{BF01}) satisfies $v_M=v_M(y_1)$ and
\begin{align*}
	0 	& = \LL^0M:M 
	   	 = \int_{Y_1} \LL(t)(M + v'_M(t)\otimes e_1):(M + v'_M(t)\otimes e_1)dt \\
		& = \int_{Y_1}\big[ \LL(t)(M+\nabla v_M(t)):(M+\nabla v_M(t)) +  dI_3:\Cof(M+\nabla v_M)
	 \big]dt.
\end{align*}
The first inequality in \eqref{0<L-Q+more} implies that the integrand of the previous expression must be pointwisely 0, and thus the inequality in \eqref{0<L-Q+more} for $\varphi=v_M$ is actually an equality. From this we deduce
\begin{equation*}
\LL M:M = Q(M),
\end{equation*}
and
\begin{equation}\label{pointwEqs}
\left\{
\begin{aligned}
	0=&\ (v_M')_1 + {l_{12}\over l_1}(v_M')_2 + {l_{13}\over l_1}(v_M')_3 + {1\over l_1}\left( \LL M:(e_1\otimes e_1) + {d\over 2}{M_{22}} + {d\over 2}{M_{33}} \right),\\
	0=&\ (v_M')_2 + {\det(\tilde{L}^{23})\over \det(\tilde{L}^{33})}(v_M')_3 - {l_{12}\over \det(\tilde{L}^{33})}\left( \LL M:(e_1\otimes e_1) + {d\over 2}{M_{22}} + {d\over 2}{M_{33}} \right) \\ 
	& + {l_1 \over \det(\tilde{L}^{33})}\left( \LL M:(e_2\otimes e_1) - {d\over 2}{M_{12}} \right),\\
	0=&\ (v_M')_3 + {\det(\tilde{L}^{13})\over \det(L)}\left( \LL M:(e_1\otimes e_1) + {d\over 2} {M_{22}} + {d\over 2}{M_{33}}\right) - {\det(\tilde{L}^{23})\over \det(L)} \left( \LL M:(e_2\otimes e_1) - {d\over 2}{M_{12}}  \right)\\
	& + {\det(\tilde{L}^{33})\over \det(L)} \left( \LL M:(e_3\otimes e_1) - {d\over 2}{M_{13}} \right).
\end{aligned}
\right.
\end{equation}
Since $v_M$ is $Y_1$-periodic, we have
\[\int_{Y_1} (v_M')_i \  dy_1 = 0\quad i=1,2,3.\]
Integrating the third equality in \eqref{pointwEqs} we obtain the first equality in \eqref{integrEqs}. Replacing $(v_M')_3$ in the second equality of \eqref{pointwEqs}, we end up getting the second equality in \eqref{integrEqs}. Finally, replacing $(v_M')_2$ and $(v_M')_3$ in the first equality of \eqref{pointwEqs} it yields the last equality in \eqref{integrEqs}.\newline\indent
Conversely, let us assume that equalities \eqref{LM:M=Q(M)} and \eqref{integrEqs} hold. Considering the first equation in \eqref{integrEqs}, taking into account that the all the integrands belong to $L^\infty(Y_1)$, there exists a function $\varphi_3\in W^{1,\infty}_{\rm per}(Y_1)$ such that, a.e. in $Y_1$, it holds
\begin{align*}
0 = &\ \varphi_3'+ {\det(\tilde{L}^{13})\over \det(L)}\left( \LL M:(e_1\otimes e_1) + {d\over 2}{M_{22}} + {d\over 2}{M_{33}} \right) - {\det(\tilde{L}^{23})\over \det(L)}\left( \LL M:(e_2\otimes e_1) - {d\over 2}{M_{12}} \right) \\
& + {\det(\tilde{L}^{33})\over \det(L)}\left( \LL M:(e_3\otimes e_1) - {d\over 2}{M_{13}} \right).
\end{align*}
Repeating the argument with the second and the third equation of \eqref{integrEqs}, we get the existence of functions $\varphi_2$ and $\varphi_1$ in $W^{1,\infty}_{\rm per}(Y_1)$ respectively, such that
\begin{align*}
&\ \varphi'_2 + {\det(\tilde{L}^{23})\over \det(\tilde{L}^{33})}\varphi'_3 - {l_{12}\over \det(\tilde{L}^{33})}\left( \LL M:(e_1\otimes e_1) + {d\over 2}{M_{22}} + {d\over 2}{M_{33}} \right) + {l_1 \over \det(\tilde{L}^{33})}\left( \LL M:(e_2\otimes e_1) - {d\over 2}{M_{12}} \right) = 0,\\
&\ \varphi'_1 + {l_{12}\over l_1}\varphi'_2 + {l_{13}\over l_1}\varphi'_3 + {1\over l_1}\left( \LL M:(e_1\otimes e_1) + {d\over 2}{M_{22}} + {d\over 2}{M_{33}} \right) = 0.
\end{align*}
These three equalities together with \eqref{LM:M=Q(M)} imply the equality in \eqref{0<L-Q+more}, and thus by \eqref{L0MM>=0} it follows that
\begin{equation*}
	0  = \int_{Y_1} \big(\LL(M+\nabla \varphi): (M + \nabla\varphi) + dI_3:\Cof(M+\nabla\varphi)
	 \big)
	 \, dy_1 \\
	  \ge \LL^0 M:M \\
	  \ge 0,
\end{equation*}
which shows that $\LL^0$ is not strongly elliptic.
\par
Finally, due to the equality $\LL^0M:M=\LL^0M^T:M^T$, conditions \eqref{LM:M=Q(M)} and \eqref{integrEqs} are equivalent to the similar equalities replacing $M$ by $M^T$.
\end{proof}
\begin{proof}[Proof of Lemma \ref{LemNoLossStrEllipIsotrop}]
Since $\LL$ is isotropic, condition \eqref{condLemmaIsotr} is equivalent to the condition \eqref{CondLM:M+D:Cof(M)>0} with $D=dI_3$. As a consequence, \eqref{condLemmaIsotr} implies $\Lambda(\LL)\ge0$. By \cite[Corollary 3.5]{GMT}, we have $\alpha_{\se}(\LL^0)\ge\Lambda(\LL)$. Therefore, we get that $\alpha_{\se}(\LL^0)\ge 0$.
\par
Assume that $\LL^0$ is not strongly elliptic, {\em i.e.} $\alpha_{\se}(\LL^0)= 0$. Then, there exists a rank-one matrix $M:=\xi \otimes \eta$ in $\mathbb{R}^{3\times 3}$, with $\xi,\eta\in\RR^3\setminus\{0\}$, such that $\LL^0M:M=0$.
\par
Since $\LL$ is isotropic, the matrix $L$ defined in \eqref{DefMatrixL} is
\begin{equation*}
L = 
\begin{pmatrix}
\lambda + 2\mu & 0 & 0\\
0 & \mu & 0 \\
0 & 0 & \mu
\end{pmatrix}.
\end{equation*}
Moreover, the following equalities hold
\begin{equation*}
\begin{array}{c}
M_{ij} =\xi_i\eta_j\quad i,j\in\{ 1,2,3 \}, \\
 \LL M:(e_1\otimes e_1)  = (\lambda + 2\mu)\xi_i\eta_1 + \lambda(\xi_2\eta_2 + \xi_3\eta_3), \\
 \LL M:(e_2\otimes e_1) = \mu(\xi_1\eta_2 + \xi_2\eta_1), \\
  \LL M:(e_3\otimes e_1) = \mu(\xi_1\eta_\cdot + \xi_3\eta_1),\\
  \LL M:M = (\lambda+\mu)(\xi:\eta)^2 + \mu|\xi|^2|\eta|^2.
 \end{array}
\end{equation*}
Because $\LL^0M:M = 0$, from equalities \eqref{LM:M=Q(M)} and \eqref{integrEqs} in Lemma \ref{LemmaAnisotropic} we obtain a.e. in $Y_1$
\begin{equation}\label{LM:M=Q(M)Isotr}
\begin{aligned}
 (\lambda+\mu)(\xi:\eta)^2 + \mu|\xi|^2|\eta|^2 =& {1\over \lambda+2\mu} \left[ (\lambda+2\mu)\xi_1\mu_1 + \lambda(\xi_2\eta_2 + \xi_3\eta_3) + {d\over2}(\xi_2\eta_2 + \xi_3\eta_3) \right]^2\\
 &+ {1\over\mu}\left[ \mu(\xi_1\eta_2 + \xi_2\eta_1) - {d\over2}\xi_1\eta_2 \right]^2 + {1\over \mu}\left[ \mu(\xi_1\eta_3 + \xi_3\eta_1) - {d\over 2}\xi_1\eta_3 \right]^2,
\end{aligned}
\end{equation}
together with
\begin{align}
0 & = \xi_1\eta_3+\xi_3\eta_1-{\xi_1\eta_3\over2}\int_{Y_1}{d\over \mu}(t)\,dt, \label{intregrEqsIsotr1} \\
0 & = \xi_1\eta_2+\xi_2\eta_1-{\xi_1\eta_2\over2}\int_{Y_1}{d\over \mu}(t)\,dt, \label{intregrEqsIsotr2} \\
0 & = \xi_1\eta_1 + (\xi_2\eta_2 + \xi_3\eta_3)\int_{Y_1} {\lambda+{d\over2} \over \lambda + 2\mu}(t)\,dt. \label{intregrEqsIsotr3}
\end{align}
After some calculations, from \eqref{LM:M=Q(M)Isotr} we get
\begin{equation}\label{LM:M=Q(M)IsotrReduced}
{(\lambda + 2\mu)^2 - (\lambda+{d\over2})^2 \over \lambda+2\mu}(\xi_2\eta_2 + \xi_3\eta_3)^2 + \mu(\xi_2\eta_3 - \xi_3\eta_2)^2 + {d(\mu-{d\over4}) \over \mu}\xi_1^2(\eta_2^2 + \eta_3^2) = 0 \quad\mbox{a.e. in }Y_1.
\end{equation}
Observe that, since $\LL$ is isotropic and (strictly) strongly elliptic in $Y_1$, we have
\[
\mu>0,\ 2\mu+\lambda>0\quad\mbox{a.e. in }Y_1,
\]
which implies that
\[
(\lambda + 2\mu)^2 - \left(\lambda+{d\over2}\right)^2\geq0\quad\mbox{a.e. in }Y_1.
\]
Hence, taking into account assumption \eqref{condLemmaIsotr}, equality \eqref{LM:M=Q(M)IsotrReduced} implies the following three conditions:
\begin{equation}\label{Cond1}
\left[ (\lambda+2\mu)^2-\left(\lambda+{d\over2}\right)^2 \right]( \xi_2\eta_2+\xi_3\eta_3 )^2 = 0  \quad\mbox{a.e. in }Y_1,
\end{equation}
\begin{equation}\label{Necessary1} 
\xi_2\eta_3 = \xi_3\eta_2,
\end{equation}
\begin{equation}\label{Cond3}
d\left( \mu - {d\over 4} \right)\xi_1^2(\eta_2^2 + \eta_3^2) = 0 \quad\mbox{a.e. in }Y_1.
\end{equation}
\par
We will now prove by contradiction that we cannot have $d=4\mu$ a.e. in $Y_1$. Otherwise, equalities \eqref{intregrEqsIsotr1}, \eqref{intregrEqsIsotr2} and \eqref{intregrEqsIsotr3} can be written as
\begin{equation}\label{3CondsContradiction}
\begin{cases}
0=\xi_1\eta_3 - \xi_3\eta_1, \\
0= \xi_1\eta_2 - \xi_2\eta_1, \\
0= \xi_1\eta_1 + \xi_2\eta_2 + \xi_3\eta_3.
\end{cases}
\end{equation}
Under these conditions, if $\eta_1\neq0$, then the first and second equalities of \eqref{3CondsContradiction} lead to
\begin{equation*}
\xi_3 = \eta_3 {\xi_1\over \eta_1},\quad \xi_2 = \eta_2 {\xi_1 \over\eta_1}.
\end{equation*}
Replacing $\xi_2$ and $\xi_3$ in the third equality in \eqref{3CondsContradiction}, we obtain
\begin{equation*}
\xi_1(\eta_1^2 + \eta_2^2 + \eta_3^2) = 0.
\end{equation*}
Since $\eta\neq 0$, we get $\xi_1 = 0$. This implies that $\xi_2=\xi_3=0$, a contradiction with $\xi\neq0$. Therefore, we have necessarily $\eta_1=0$. Moreover, using the two first equalities of \eqref{3CondsContradiction} and the fact that $\eta\neq0$, we obtain $\xi_1=0$. As a consequence, \eqref{3CondsContradiction} reduces to
\begin{equation}\label{3CondsContradReduced}
\xi_2\eta_2 + \xi_3\eta_3 = 0.
\end{equation}
 If $\eta_2\neq 0$, then using \eqref{Necessary1} we get
 \begin{equation*}
 \xi_3 = \xi_2 {\eta_3\over \eta_2},
 \end{equation*}
 and replacing $\xi_3$ in the previous equality, it yields
\begin{equation*}
\xi_2(\eta_2^2 + \eta_3^2) = 0.
\end{equation*}
Again, since $\eta\neq0$, we have $\xi_2 = 0$. Using \eqref{Necessary1} and the assumption $\eta_2 \neq 0$, it follows that $\xi_3 = 0$, again a contradiction with $\xi,\eta\neq 0$. Thus, we have necessarily $\eta_2 = 0$. Taking into account that $\eta_1=\eta_2=0$ we have $\eta_3\neq 0$, hence from \eqref{3CondsContradReduced} we deduce that $\xi_3= 0$. Now \eqref{Necessary1} is written as $\xi_2\eta_3=0$. However, recall that $\xi_1=\xi_3=\eta_1=\eta_2=0$. This implies that either $\xi=0$ or $\eta=0$, a contradiction.
\par
We have just shown that the set $\{ d < 4\mu\}$ has a positive Lebesgue measure. Similarly, we can check that~$d>0$. Using \eqref{Cond1} and \eqref{Cond3} together with $0<d\leq 4\mu$, we deduce that
\begin{equation*}
\xi_2\eta_2+\xi_3\eta_3 =\xi_1^2(\eta_2^2 + \eta_3^2) = 0,
\end{equation*}
which combined with \eqref{intregrEqsIsotr3} also gives $\xi_1\eta_1=0$.
As above, using the three previous equalities, \eqref{intregrEqsIsotr1}, \eqref{intregrEqsIsotr2} and \eqref{Necessary1}, we get a contradiction with the fact that $\xi,\eta\neq 0$. Therefore, we have proved that $\LL^0$ is strongly elliptic if \eqref{condLemmaIsotr} holds for some~$d$.
\end{proof}

\subsection{Rank-two lamination}\label{SubSecRank2Lam}

In the proof of Proposition \ref{GutProp} for dimension three \cite[Section 5.2]{G01}, Guti\'errez performed a rank-one laminate mixing a strongly elliptic but not semi-very strongly isotropic material $\LL_a$, and a very strongly elliptic isotropic material $\LL_b$. However, as it was noted at the beginning of the section, there are some cases for which the strong ellipticity of the homogenized tensor is not lost after this first lamination. In fact, our Theorem~\ref{ThNoLossFirstLaminate} shows that for a general rank-one laminate, it is not possible to lose the strong ellipticity through homogenization if there exists a matrix $D\in \mathbb{R}^{3\times 3}$ satisfying condition \eqref{CondLM:M+D:Cof(M)>0}. As done in \cite{G01}, we need to perform a second lamination with a third material $\LL_c$ which can be very strongly elliptic, in order to lose the strong ellipticity in those cases.
\par
Our purpose is to justify Guti\'errez' approach using formally $1^*$-convergence (see \cite[Section~3]{G01}), by a homogenization procedure using the $\Gamma$-convergence result of Theorem \ref{BBResult}.
\begin{theorem}\label{ThRankTwoLossStrEll}
For any strongly elliptic but not semi-very strongly elliptic isotropic tensor $\LL_a$ whose Lam\'e coefficients satisfy
\begin{equation}\label{CondMaterialA}
4\mu_a + 3\lambda_a > 0,
\end{equation}
there exist two very strongly elliptic isotropic tensors $\LL_b,\LL_c$ and volume fractions $\theta_1,\theta_2\in(0,1)$ such that the tensor $\LL_2$ obtained by laminating in the direction $y_2$ the effective tensor $\LL_1^*$ -- firstly obtained by laminating in the direction $y_1$ the tensors $\LL_a$, $\LL_b$ with proportions $\theta_1$, $1-\theta_1$ -- and the tensor $\LL_c$ with proportions $\theta_2$ and $1-\theta_2$ respectively, namely 
\begin{equation}\label{DefL2}
\LL_2(y_2):=\chi_2(y_2)\,\LL_1^* + \big(1-\chi_2(y_2)\big)\,\LL_c \quad\mbox{for } y_2\in Y_1,
\end{equation}
satisfies
\begin{equation}\label{Lambda(L2)=0}
\Lambda(\LL_2)=0,
\end{equation}
and
\begin{equation}\label{GammaLimRankTwo}
\int_\Omega \LL_2(x_2/\varepsilon)\nabla v : \nabla v\, dx \stackrel{\Gamma-H^1_0(\Omega)^3}\rightharpoonup \int_\Omega \LL_2^0\nabla v : \nabla v\, dx,
\end{equation}
where the homogenized tensor $\LL_2^0$ is not strongly elliptic, {\em i.e.}
\begin{equation}\label{LossOfEllipticityRankTwo}
\alpha_{\se}(\LL_2^0) = 0.
\end{equation}
\end{theorem}

\begin{remark}
Theorem \ref{ThRankTwoLossStrEll} shows that for certain strongly elliptic but not very strongly elliptic isotropic tensors, namely those whose Lam\'e parameters satisfy \eqref{CondMaterialA}, it is possible to find two very strongly elliptic isotropic tensors for which the homogenization process through $\Gamma$-convergence using a rank-two lamination leads to the loss of ellipticity of the effective tensor.
\end{remark}

\begin{proof}[Proof of Theorem \ref{ThRankTwoLossStrEll}]
We divide the proof into four steps.
\par\medskip\noindent
{\it Step 1.} Choice of $\LL_a$, $\LL_b$, $\theta_1$, $\theta_2$.
\par\noindent
Let $\LL_a$ be a strongly elliptic but not semi-very strongly elliptic isotropic tensor satisfying~\eqref{CondMaterialA}. Our aim is to find two very strongly isotropic tensors $\LL_b,\LL_c$ and two volume fractions $\theta_1,\theta_2$ such that the strong ellipticity is lost through homogenization using a rank-two lamination.
\par
Let $\chi_1, \chi_2:\RR\to\{0,1\}$ be two $1$-periodic characteristic functions such that
\[
\int_{Y_1}\chi_1(y_1)\, dy_1 = \theta_1\quad\mbox{and}\quad \int_{Y_1}\chi_2(y_2)\, dy_2 = \theta_2,
\]
where $\theta_1,\theta_2 \in (0, 1)$ will be chosen later.
\par
The $1^*$-convergence procedure of \cite[Section 5.2]{G01} applied to the tensor
\begin{equation}\label{FirstLam}
\LL_1(y_1):=\chi_1({y_1})\,\LL_a + \big(1-\chi_1({y_1})\big)\,\LL_b\quad\mbox{for }y_1\in Y_1,
\end{equation}
yields a non-isotropic effective tensor $\LL_1^*$. The computations of \cite[Section 5.2]{G01} lead to an explicit expression of the tensor $\LL_1^*$ whose non-zero entries are
\begin{equation}\label{ExpressionOfL1*}
\begin{aligned}
&(\LL_1^*)_{1111}={1\over A},\\
&(\LL_1^*)_{1122} = (\LL_1^*)_{2211} = (\LL_1^*)_{1133} = (\LL_1^*)_{3311} = {B\over A},\\
&(\LL_1^*)_{1212} = (\LL_1^*)_{1221} = (\LL_1^*)_{2112} = (\LL_1^*)_{2121} = {1\over E},\\
&(\LL_1^*)_{1313} = (\LL_1^*)_{1331} = (\LL_1^*)_{3113} = (\LL_1^*)_{3131} = {1\over E},\\
&(\LL_1^*)_{2222} = {B^2 \over A} + 2(C+D),\\
& (\LL_1^*)_{2233} = (\LL_1^*)_{3322} = {B^2\over A} + 2D,\\
&(\LL_1^*)_{2323} = (\LL_1^*)_{2332} = (\LL_1^*)_{3223} = (\LL_1^*)_{3232} = C,\\
&(\LL_1^*)_{3333} = {B^2\over A} + 2(C+D),
\end{aligned}
\end{equation}
where 
\begin{equation}\label{DefABCDE}
\begin{aligned}
& A = {\theta_1\over 2\mu_a + \lambda_a} + {1-\theta_1 \over 2\mu_b + \lambda_b},\\
& B = {\theta_1 \lambda_a \over 2\mu_a + \lambda_a} + {(1-\theta_1)\lambda_b\over 2\mu_b + \lambda_b},\\
& C = \theta_1\mu_a + (1-\theta_1)\mu_b,\\
& D = {\theta_1 \mu_a\lambda_a\over 2 \mu_a + \lambda_a} + {(1-\theta_1)\mu_b\lambda_b\over 2\mu_b + \lambda_b},\\
& E = {\theta_1\over \mu_a} + {1-\theta_1\over \mu_b}.
\end{aligned}
\end{equation}
\par
Now, let us specify the choice of the two very strongly elliptic isotropic tensors $\LL_b$, $\LL_c$, and the volume fractions $\theta_1$, $\theta_2$. For the Lam\'e parameters of material $c$ we denote $\lambda_c = \alpha_c \mu_c$ as done in \cite{G01}. We assume that
\begin{equation}\label{CondMuB}
-{1\over 4}(2\mu_a + 3\lambda_a) \le \mu_b < {\mu_a(2\mu_a + 3\lambda_a) \over 3\lambda_a},
\end{equation}
\begin{equation}\label{CondLamdaB}
\lambda_b > {2\mu_b^2\lambda_a \over \mu_a(2\mu_a + 3\lambda_a) - 3\mu_b\lambda_a} ,
\end{equation}
\begin{equation}\label{CondTheta1}
\theta_1 = {-\lambda_b(2\mu_a + \lambda_a)\over 2(\mu_b\lambda_a - \mu_a\lambda_b)},
\end{equation}
\begin{equation}\label{CondAlphaC}
\alpha_c\ge {-D \over C+D},
\end{equation}
\begin{equation}\label{CondMuC}
\mu_c = C {\alpha_c (C+2D) \over D(1 + \alpha_c)},
\end{equation}
and
\begin{equation}\label{CondTheta2}
\theta_2 = {\alpha_c(C+D) \over \alpha_c(C+D) - D(2+\alpha_c)}.
\end{equation}
 Observe that, thanks to the first inequality in \eqref{CondMuB}, the tensor $\LL_1$ given by \eqref{FirstLam} satisfies $\Lambda(\LL_1)\ge0$ (see \cite[Section 4.2]{G01}). Hence, by Theorem \ref{ThNoLossFirstLaminate} the homogenized tensor $\LL_1^*$ is strongly elliptic. This justifies the first lamination from the point of view of homogenization through $\Gamma$-convergence.
 \par
To conclude the first step, let us check that the previous conditions satisfy the assumptions of Theorem \ref{ThRankTwoLossStrEll}.
The tensor $\LL_a$ is strongly elliptic but not semi-very strongly elliptic, {\em i.e.}
\begin{equation*}
\mu_a > 0, \quad 2\mu_a + 3\lambda_a < 0,
\end{equation*} 
which implies that $\mu_b>0$. The fact that necessarily $\lambda_a<0$ together with \eqref{CondMuB} implies that $\lambda_b>0$ thanks to \eqref{CondLamdaB}, and thus $\LL_b$ is very strongly elliptic.
The volume fraction $\theta_1$ clearly belongs to $(0,1)$, since \eqref{CondTheta1} reads~as
\begin{equation*}
\theta_1 = {\lambda_b(2\mu_a + \lambda_a) \over \lambda_b(2\mu_a + \lambda_a) - \lambda_a(2\mu_b + \lambda_b)}.
\end{equation*}
The choice of $\theta_1$ implies that in \eqref{DefABCDE}
\begin{equation}\label{B=0}
B=0.
\end{equation}
In addition, $C+D>0$ as it was proved in \cite[Appendix C]{G01} and $C+2D<0$ by \eqref{CondMuB}, \eqref{CondLamdaB} and \eqref{CondTheta1}. This also implies that $D<0$. Thanks to the previous inequalities we have $\theta_2\in(0,1)$, $\alpha_c>0$ and $\mu_c>0$, which implies that $\LL_c$ is very strongly elliptic.
\par\medskip\noindent
{\it Step 2.} $\Lambda(\LL_2)\ge0$.\label{PropLambda(L2)>=0} 
\par\noindent
To get $\Lambda(\LL_2)\ge0$ we will prove that for
\begin{equation*}
D:=
\begin{pmatrix}
4\mu_c & 0\ & 0\ \\
0 & 0\ & 0\ \\
0 & 0\ & 0\
\end{pmatrix},
\end{equation*}
we have
\begin{equation}\label{LambdaL2>=0}
\LL_2(y_2)M:M + D:\Cof(M)\ge0\quad\mbox{a.e. }y_2\in Y_1,\ \mbox{for all }M\in\mathbb{R}^{N\times N}.
\end{equation}
We need to prove that the previous inequality holds in each homogeneous phase of $\LL_2$.
\par
Firstly, for the phase $\LL_c$ which is isotropic and very strongly elliptic, we get for any $M\in\mathbb{R}^{3\times 3}$,
\begin{align*}
&\ \LL_cM:M + D:\Cof(M) \\
			 = &\  2\mu_c\left[M_{11}^2 + M_{22}^2 + M_{33}^2 + 2\left(M_{12} + M_{21} \over 2 \right)^2 + 2\left(M_{13} + M_{31} \over 2 \right)^2 + 2\left(M_{23} + M_{32} \over 2 \right)^2\right] \\
			 & + \lambda_c(M_{11} + M_{22} + M_{33})^2 + 4\mu_c(M_{22}M_{33} - M_{23}M_{32}) \\
			 = & \ (\lambda_c + 2\mu_c)(M_{11}^2 + M_{22}^2 + M_{33}^2) + 2\lambda_c(M_{11}M_{22} + M_{11}M_{33}) + 2(\lambda_c + 2\mu_c)M_{22}M_{33}  \\
			 & + \mu_c(M_{12}+ M_{21})^2+\mu_c(M_{31}+M_{13})^2+\mu_c(M_{23}-M_{32})^2. 
\end{align*}
This quantity is non-negative for any $M\in\mathbb{R}^{3\times3}$, since the following matrix is positive semi-definite:
\begin{equation*}\label{MatrixPinLc}
\begin{pmatrix}
\lambda_c + 2\mu_c & \lambda_c & \lambda_c \\
\lambda_c & \lambda_c + 2\mu_c & \lambda_c + 2\mu_c \\
\lambda_c & \lambda_c + 2\mu_c & \lambda_c + 2\mu_c
\end{pmatrix},
\end{equation*}
due to the strong ellipticity of $\LL_c$. Therefore, the desired inequality holds for the homogeneous phase $\LL_{c}$.
\par
Secondly, we need to check the same inequality for the phase with $\LL_1^*$. By \eqref{ExpressionOfL1*} we have for $M\in\mathbb{R}^{3\times 3}$,
\begin{equation*}
\begin{aligned}
\LL_1^*M:M + D:\Cof(M) 
				= & \ {1\over A}M_{11}^2 + \left[ {B^2\over A}+2(C+D) \right](M_{22}^2 + M_{33}^2) + 2{B\over A}(M_{11}M_{22} + M_{11}M_{33}) \\
				& +2\left[ {B^2\over A} + 2D + 2\mu_c \right](M_{22}M_{33}) \\
				& + {1\over E}(M_{12}+M_{21})^2+{1\over E}(M_{13}+M_{31})^2 \\
				& + C(M_{23}^2 + M_{32}^2) + 2(C-2\mu_c)M_{23}M_{32}.
\end{aligned}
\end{equation*}
Since $E\ge0$, this quantity is non-negative for any $M\in\mathbb{R}^{3\times 3}$ if the following two matrices are positive semi-definite:
\begin{equation}\label{MatrixPinL1*}
\begin{pmatrix}
{1\over A} & {B\over A} & {B\over A} \\ \\
{B\over A} & {B^2\over A} + 2(C+D) & {B^2\over A} + 2D + 2\mu_c \\ \\
{B\over A} & {B^2\over A} + 2D + 2\mu_c & {B^2\over A} + 2(C+D)
\end{pmatrix},
\end{equation}
\begin{equation}\label{MatrixQ2inL1*}
\begin{pmatrix}
C & C-2\mu_c\\ \\
C-2\mu_c & C
\end{pmatrix}.
\end{equation}
Since $C\ge0$, the matrix \eqref{MatrixQ2inL1*} is positive semi-definite if and only if $\mu_c \le C$.
Taking into account that $\mu_c \le C$, we can check that the matrix \eqref{MatrixPinL1*} is positive semi-definite if  $-(C+2D)\le \mu_c$.
Therefore, the matrices \eqref{MatrixPinL1*} and \eqref{MatrixQ2inL1*} are positive semi-definite if
\begin{equation}\label{Cond-(C+2D)<Mu_c<C}
-(C+2D)\le \mu_c\leq C.
\end{equation}
By the definition \eqref{CondMuC} of $\mu_c$, we deduce that the first inequality of \eqref{Cond-(C+2D)<Mu_c<C} holds if and only if
\begin{equation*}
{\alpha_c C \over -D(1+\alpha_c)}\ge1,
\end{equation*} 
which is satisfied due to inequality \eqref{CondAlphaC}.
For the second inequality of \eqref{Cond-(C+2D)<Mu_c<C}, we need to check that (see \eqref{CondMuC})
\begin{equation*}
{\alpha_c(C+2D)\over D(1+\alpha_c)}\le1,
\end{equation*}
or equivalently,
\begin{equation*}
\alpha_c\ge{D\over C+D}.
\end{equation*}
This is true since $\alpha_c>0$ by \eqref{CondAlphaC} and ${D\over C+D}<0$. Therefore, condition \eqref{LambdaL2>=0} holds true, and consequently
\begin{equation}\label{Lambda(L2)>=0}
\Lambda(\LL_2)\ge0.
\end{equation}
\noindent
{\it Step 3.} $\LL_2$ loses the strong ellipticity through homogenization.
\par\noindent
On the one hand, due to $\Lambda(\LL_2) \ge 0$, by virtue of Theorem \ref{BBResult} the $\Gamma$-convergence \eqref{GammaLimRankTwo} holds with the homogenized tensor $\LL_2^0$ which is given by the minimization formula \eqref{DefL0} replacing $\LL$ by $\LL_2$.
\par
On the other hand, following Guti\'errez' $1^*$-convergence procedure we obtain a homogenized tensor $\LL_2^*$ such that (see \cite[Section 5.2]{G01} for the expression of $\LL_2^*$)
\begin{equation*}
\LL_2^*(e_3\otimes e_3):(e_3\otimes e_3) =  I_1 + {G_1^2\over F_1},
\end{equation*}
where by \eqref{B=0},
\begin{align*}
I_1 &  = 4(1-\theta_2){1+\alpha_c\over 2+\alpha_c} + 2\theta_2C{C+2D\over C+D}, \\
G_1&  = (1-\theta_2){\alpha_c\over 2+\alpha_c} + \theta_2{D\over C+D},\\
F_1&  \neq 0.
\end{align*}
It is not difficult to check that the choice of $\LL_b$, $\LL_c$, $\theta_1$, $\theta_2$ leads to $I_1 = G_1 = 0$, which yields
\begin{equation}\label{LossEllipticityL_2^0Proof}
\LL_2^*(e_3\otimes e_3):(e_3\otimes e_3) = 0.
\end{equation}
\par
To conclude the proof it is enough to show that
\begin{equation}\label{L_2^*=L_2^0}
\LL_2^* = \LL_2^0.
\end{equation}
Indeed, thanks to $\LL_2^* = \LL_2^0$ equality \eqref{LossEllipticityL_2^0Proof} implies the loss of ellipticity \eqref{LossOfEllipticityRankTwo}, and \eqref{LossOfEllipticityRankTwo} implies $\Lambda(\LL_2)\le0$. This combined with \eqref{Lambda(L2)>=0} finally shows the desired lost of functional coercivity \eqref{Lambda(L2)=0}.
\par\medskip\noindent
{\it Step 4.} $\LL_2^* = \LL_2^0$.
\par\noindent
By formally using $1^*$-convergence in terms of \cite[Lemma 3.1]{BF01}, Guti\'errez's computations for the tensor $\LL_2^*$ in \cite[Section 5.2]{G01} can be written as
\begin{equation}\label{FormulaeForL_2^*}
\left\{\begin{aligned}
&A^{-1}[\LL_2^*] = \int_0^1 A^{-1}[\LL_2](t)\, dt,\\
&A^{-1}_{im}[\LL_2^*](\LL_2^*)_{2mkl}=\int_{0}^1\big( A^{-1}_{im}[\LL_2](t)(\LL_2)_{2mkl}(t) \big)\,dt,\\
&(\LL_2^*)_{ijkl} - (\LL_2^*)_{ij2m}A^{-1}_{mn}[\LL_2^*](\LL_2^*)_{2nkl} = \int_{0}^1\big( (\LL_2)_{ijkl}(t) - (\LL_2)_{ij2m}(t)A^{-1}_{mn}[\LL_2](t)(\LL_2)_{2nkl}(t) \big)\,dt,
\end{aligned}
\right.
\end{equation}
where in the present context, for any $\LL\in L^\infty_{\per}(Y_1;\mathscr{L}_s(\mathbb{R}^{3\times 3}))$, $A[\LL]\in L^\infty_{\per}(Y_1;\mathbb{R}^{3\times 3}_s)$ is defined by 
\begin{equation*}
A[\LL](y_2)\xi := [\LL(y_2)(\xi\otimes e_2)]e_2\quad\mbox{for }y_2\in Y_1\mbox{ and }\xi\in \mathbb{R}^{3}.
\end{equation*}
By focusing on the first equality of \eqref{FormulaeForL_2^*} we have
\begin{equation}\label{FirstEqForL_2^*}
A^{-1}[\LL_2^*] = \int_0^1 A^{-1}[\LL_2](t)\, dt = \theta_2 A^{-1}[\LL_1^*] + (1-\theta_2) A^{-1}[\LL_c],
\end{equation}
where all the quantities are finite. Now, similarly to the proof of Theorem \ref{BBResult} we consider the perturbation of $\LL_2$ defined by
\begin{equation}
\LL_\delta:=	\LL_2 + \delta\,\mathbb{I}_s\quad\mbox{for }\delta>0.
\end{equation}
On the one hand, due to $\Lambda(\LL_\delta)>0$ (which by \eqref{inecoer} implies $0<\Lambda_{\per}(\LL_\delta)\le\alpha_{\se}(\LL_\delta)$), thanks to \cite[Lemma 3.2]{BF01} the $1^*$-limit $\LL_\delta^*$ of $\LL_\delta$ and the homogenized tensor $\LL_\delta^0$ of $\LL_\delta$ defined by \eqref{DefL0} agree.
Then, applying \cite[Lemma 3.1]{BF01} with $\LL_\delta$ we get that
\begin{equation}\label{FirstEqForL_delta^*}
A^{-1}[\LL_\delta^*] = \int_0^1 A^{-1}[\LL_\delta](t) dt = \theta_2 A^{-1}[\LL_1^*+\delta\,\mathbb{I}_{s}] + (1-\theta_2) A^{-1}[\LL_c+\delta\,\mathbb{I}_{s}].
\end{equation}
Observe that we have
\begin{equation*}
\begin{aligned}
A[\LL_1^* + \delta\,\mathbb{I}_s] & \ge A[\LL_1^*], \\
A[\LL_1^* + \delta\,\mathbb{I}_s] & \to A[\LL_1^*]\quad \mbox{as }\delta\to0, 
\end{aligned}
\end{equation*}
where the previous inequality must be understood in the sense of the quadratic forms.
This combined with the fact that both $\LL_1^* + \delta\,\mathbb{I}_s$ and $\LL_1^*$ are strongly elliptic tensors (which implies that the previous matrices are positive definite), yields
\begin{equation*}
A^{-1}[\LL_1^* + \delta\,\mathbb{I}_s] \le A^{-1}[\LL_1^*],
\end{equation*}
and thus, 
\begin{equation*}
A^{-1}[\LL_1^* + \delta\,\mathbb{I}_s]  \to A^{-1}[\LL_1^*]\quad \mbox{as }\delta\to0. 
\end{equation*}
Similarly, we have
\begin{equation*}
A^{-1}[\LL_c + \delta\,\mathbb{I}_s]  \to A^{-1}[\LL_c]\quad \mbox{as }\delta\to0. 
\end{equation*}
Hence, from the two previous convergences and taking into account \eqref{FirstEqForL_2^*}, \eqref{FirstEqForL_delta^*}, we deduce that
\begin{equation*}
A^{-1}[\LL_\delta^*]\to A^{-1}[\LL_2^*]\quad\mbox{as }\delta\to0.
\end{equation*}
On the other hand, following the proof of Theorem \ref{BBResult} we have
\[
\LL_\delta^* = \LL_\delta^0\to\LL_2^0\quad \mbox{as }\delta\to0.
\]
Therefore, we obtain the equality
\begin{equation}\label{FirstEqForL_delta^*andL_delta^0}
A^{-1}[\LL_2^0]= A^{-1}[\LL_2^*].
\end{equation}
Using similar arguments, we can prove that $\LL_2^0$ and $\LL_2^*$ satisfy for any $i,j,k,l\in\{1,2,3\}$,
\begin{align}
 A^{-1}_{im}[\LL_2^*](\LL_2^*)_{2mkl} & = A^{-1}_{im}[\LL_2^0](\LL_2^{0})_{2mkl}\label{SecondEqForL_delta^*andL_delta^0},\\
 (\LL_2^*)_{ijkl} - (\LL_2^*)_{ij2m}A^{-1}_{mn}[\LL_2^*](\LL_2^*)_{2nkl} & = (\LL_2^0)_{ijkl} - (\LL_2^0)_{ij2m}A^{-1}_{mn}[\LL_2^0](\LL_2^0)_{2nkl}. \label{ThirdEqForL_delta^*andL_delta^0}
\end{align}
Since the set of equalities \eqref{FormulaeForL_2^*} completely determine the tensor $\LL_2^*$,
equalities \eqref{FirstEqForL_delta^*andL_delta^0}, \eqref{SecondEqForL_delta^*andL_delta^0}, \eqref{ThirdEqForL_delta^*andL_delta^0} thus imply the desired equality \eqref{L_2^*=L_2^0}, which concludes the proof.
\end{proof}

\section*{Appendix}

\begin{proof}[Proof of Theorem \ref{LamdaPer>0}]
We simply adapt the proof of \cite[Theorem 2.2]{BF01} to dimension 3.
\par
Firstly, let us prove the first part of the theorem, {\em i.e.} $\Lambda(\LL)\ge0$.
The quasi-affinity of the cofactors (see \cite{Dac}) reads as
\begin{equation}\label{null-Lagr}
\int_{Y_3} \adj_{ii}(\nabla v)\,dy = 0,\quad \forall\, v\in C_c^\infty (\mathbb{R}^3;\mathbb{R}^3),\  \forall\,i\in\{1,2,3\}.
\end{equation}
As a consequence, for any $d\in\mathbb{R}$, the definition of $\Lambda(\LL)$ can be rewritten as
\begin{equation*}
\Lambda (\LL) = \inf \left\{ \int_{\mathbb{R}^3} \left[ \LL e(v) : e(v) + d \sum_{i=1}^3 \adj_{ii}(\nabla v)\right]dy
,\  {v\in C_c^\infty (\mathbb{R}^3;\mathbb{R}^3)} \right\}.
\end{equation*}
If we compute the integrand in the previous infimum, we obtain
\begin{equation}\label{LambdaLRewritten}
\Lambda (\LL) = \inf \left\{ \int_{\mathbb{R}^3} \left[ P(y;\partial_1v_1,\partial_2v_2,\partial_3v_3)+ Q(y;\partial_3v_2,\partial_2v_3) + Q(y;\partial_3v_1,\partial_1v_3) + Q(\partial_2v_1,\partial_1v_2)\right] dy,\  {v\in C_c^\infty (\mathbb{R}^3;\mathbb{R}^3)}\right\},
\end{equation}
where
\[
\left\{
\begin{array}{l}
	P(y;a,b,c) := 
	\begin{pmatrix}
	a & b & c
	\end{pmatrix}
	\begin{pmatrix}
	\lambda + 2\mu & \lambda + \frac{d}{2} & \lambda + \frac{d}{2} \\
	\lambda + \frac{d}{2} & \lambda + 2\mu & \lambda + \frac{d}{2} \\
	\lambda + \frac{d}{2} & \lambda + \frac{d}{2} & \lambda + 2\mu
	\end{pmatrix}
	\begin{pmatrix}
	a \\ b \\c
	\end{pmatrix},
	\\ \\
	Q(y;a,b) := 
	\begin{pmatrix}
	a & b
	\end{pmatrix}
	\begin{pmatrix}
	\mu & \mu-\frac{d}{2} \\
	\mu-\frac{d}{2} & \mu
	\end{pmatrix}
	\begin{pmatrix}
	a \\ b
	\end{pmatrix}.
\end{array}
\right.
\]
We can check that condition \eqref{HipExConstd} with $d\geq 0$ implies that the quadratic forms $P$ and $Q$ are non negative.
Hence, the integrand in \eqref{LambdaLRewritten} is pointwisely non-negative, and thus $\Lambda(\LL)\ge0$.
\par
Now, let us prove that $\Lambda_{\per}(\LL)> 0$. By the definition of $\Lambda_{\per}(\LL)$ and using the same argument as before, we have
\begin{equation*}
\Lambda_{\per} (\LL) = \inf \left\{ \int_{Y_3} \left[ \LL e(v) : e(v) + d \sum_{i}\adj_{ii}(\nabla v)
\right]dy, \ v\in H^1_{\per}(Y_3;\mathbb{R}^3),\int_{Y_3}|\nabla v|^2\, dy = 1
 \right\}.
\end{equation*}
Similar computations lead to
\begin{equation}\label{LambdaPerEq}
\Lambda_{\per} (\LL) = \inf \left\{ 
\int_{Y_3 }
\big[ P(y;\partial_1v_1,\partial_2v_2,\partial_3v_3)+ Q(y;\partial_3v_2,\partial_2v_3) + Q(y;\partial_3v_1,\partial_1v_3) + Q(\partial_2v_1,\partial_1v_2) \big] dy 
\right\}.
\end{equation}
Take $y\in Z_i, i\in I$. Then, using that $4\mu_i = d$, we have
\begin{equation*}
P(y;a,b,c) = (\lambda_i + 2\mu_i)(a+b+c)^2 \geq 0,
\end{equation*}
and
\begin{equation*}
Q(y;a,b) = \mu_i(a-b)^2 \geq 0.
\end{equation*}
For $y\in Z_j, j\in J$, using that $2\mu_j + 3\lambda_j = -d$, we get
\begin{equation*}
P(y;a,b,c) = \left(\mu_j + \frac{\lambda_j}{2}\right)\big[(a-b)^2+(a-c)^2+(b-c)^2\big] \geq 0,
\end{equation*}
and
\begin{equation*}
Q(y;a,b) = d\left(\mu_j + \frac{d}{4}\right)\geq  0.
\end{equation*}
Finally, for $y\in Z_k$, $k\in K$, since $-(2\mu_k+3\lambda_k) < d < 4\mu_k$, it is easy to see that the quadratic forms $P$ and $Q$ are positive semi-definite.
Therefore, we have just proved that there exists $\alpha>0$ such that
\begin{align}
&P(y;a,b,c) \ge \alpha(a+b+c)^2,\ Q(y;a,b) \ge \alpha(a-b)^2,\  y\in Z_i,i\in I,\label{ineqInI}\\
&P(y;a,b,c) \ge \alpha[(a-b)^2 + (a-c)^2 + (b-c)^2],\ Q(y;a,b) \ge \alpha(a^2+b^2),\ y\in Z_j,j\in J,\label{ineqInJ}\\
&P(y;a,b,c) \ge \alpha(a^2 + b^2 + c^2),\ Q(y;a,b) \ge \alpha(a^2 + b^2),\ y\in Z_k,k\in K, \label{ineqInK}
\end{align}
which implies that $\Lambda_{\per}(\LL) \geq 0$.
\par
Assume by contradiction that $\Lambda_{\per}(\LL)=0$. In this case there exists a sequence $v^n\in H^1_{\per}(Y_3;\mathbb{R}^3)$ with 
\[
\int_{Y_3} v^n \, dy = 0,
\]
such that
\begin{equation}\label{hipRedAbs}
\int_{Y_3}|\nabla v^n|^2 \, dy = 1,\quad \forall\, n\in\mathbb{N},
\end{equation}
together with 
\begin{equation*}
\int_{Y_3}\LL(y)e(v^n) : e(v^n) \, dy \to 0.
\end{equation*}
By the Poincar\'e-Wirtinger inequality $v^n$ is bounded in $L^2(Y_3;\mathbb{R}^3)$. Moreover, by \eqref{LambdaPerEq} we have
\begin{equation}\label{tendsToZeroY3}
\int_{Y_3}\left[ P(y;\partial_1 v_1^n,\partial_2 v_2^n,\partial_3 v_3^n) + \sum_{i<j} Q(y;\partial_j v_i^n,\partial_i v_j^n) \right]dy \to 0.
\end{equation}

Take $k\in K$. Using \eqref{ineqInK} we get 
\[
\int_{Z_k}\left[ P(y;\partial_1 v_1^n,\partial_2 v_2^n,\partial_3 v_3^n) + \sum_{i<j} Q(y;\partial_j v_i^n,\partial_i v_j^n) \right]dy \ge \alpha\int_{Z_k} |\nabla v^n|^2\,dy.
\]
Then, using \eqref{tendsToZeroY3} and the fact that both $P$ and $Q$ are non negative, it follows that
\begin{equation*}
\int_{Z_k} |\nabla v^n|^2\,dy \to 0\quad\forall\, k\in K,
\end{equation*}
and therefore
\begin{equation}\label{lim0inK}
 \lim_{n\to\infty} \sum_{k\in K}\int_{Z_k}\sum_{q,r=1,2,3} (\partial_r v_q^n)^2 \, dy = 0.
\end{equation}
Next, take $j\in J$. By \eqref{ineqInJ} we obtain
\[
	\int_{Z_j}\left[ P(y;\partial_1 v_1^n,\partial_2 v_2^n,\partial_3 v_3^n) + \sum_{i<k} Q(y;\partial_k v_i^n,\partial_i v_k^n) \right]dy \ge \alpha\int_{Z_j} \sum_{i<k} \left[ (\partial_i v_i^n - \partial_k v_k^n)^2 + (\partial_k v_i^n)^2 + (\partial_i v_k^n)^2 \right]dy.
\]
Again using \eqref{tendsToZeroY3} and the non-negativity of $P$ and $Q$ we get
\begin{equation}\label{tendsToZeroZj}
 \lim_{n\to\infty} \int_{Z_j} \left[ (\partial_i v_i^n - \partial_k v_k^n)^2 + (\partial_k v_i^n)^2 + (\partial_i v_k^n)^2 \right] = 0
 \quad\mbox{for }i,k\in\{1,2,3\},\ i<k.
\end{equation}
From \eqref{tendsToZeroZj} and the continuity of the operator $\partial_1 : L^2(Z_j) \to H^{-1}(Z_j)$ we deduce that
\begin{equation}\label{dd1v1convH-1}
\left\{
\begin{aligned}
\partial_2(\partial_1 v_1^n)=\partial_1(\partial_2 v_1^n) &\to 0\quad\mbox{strongly in }H^{-1}(Z_j),\\
\partial_1(\partial_1 v_1^n)=\partial_1(\partial_1 v_1^n-\partial_2 v_2^n)+\partial_2(\partial_1 v_2^n) &\to 0\quad\mbox{strongly in }H^{-1}(Z_j).
\end{aligned}
\right.
\end{equation}
By \eqref{hipRedAbs} we also have
\begin{equation}\label{d1v1acotL2}
\partial_1 v_1^n \mbox{ is bounded in }L^2(Z_j).
\end{equation}
However, thanks to Korn's Lemma (see, {\em e.g.}, \cite{Necas}) the following norms are equivalent in $L^2(Z_j)$:
	\begin{equation*}
	\left\{
	\begin{aligned}
	&\|\nabla \cdot \|_{H^{-1}(Z_j;\mathbb{R}^3)} + \| \cdot \|_{H^{-1}(Z_j)},\\
	&\| \cdot \|_{L^2(Z_j)}.
	\end{aligned}
	\right.
	\end{equation*}
Hence, from estimates \eqref{dd1v1convH-1}, \eqref{d1v1acotL2} and the compact embedding of $L^2$ into $H^{-1}$, it follows that 
\[
\partial_1 v_1^n \mbox{ is strongly convergent in } L^2(Z_j).
\]
Furthermore, by \eqref{dd1v1convH-1} and the fact that $Z_j$ is connected for all $j$, there exists $c_j\in \mathbb{R}$ such that 
\[
\partial_1 v_1^n \to c_j \mbox{ strongly in $L^2(Z_j)$},
\]
which combined with \eqref{tendsToZeroZj} yields
\begin{equation*}
\nabla v^n \to c_j I_3\mbox{ strongly in }L^2(Z_j)^3.
\end{equation*}
Since $v^n$ is bounded in $L^2(Y_3;\mathbb{R}^3)$, we can conclude that there exists $V_j\in \mathbb{R}^3$ such that
\begin{equation}
v^n\to v:=c_j y + V_j\quad\mbox{strongly in }H^{-1}(Z_j;\mathbb{R}^3).
\end{equation}
\par
In Case 1, by the periodicity of the limit $c_j y + V_j$ it is necessary to have $c_j=0$.
\par
In Case 2, since $Z_k$ is connected, by \eqref{lim0inK} there exists a constant $c_k$ such that $v_n$ converges to $\chi_{Z_j}v + \chi_{Z_k}c_k$ strongly in $H^1(Z_j\cup Z_k)$. Hence, since the sets $Z_j$ and $Z_k$ are regular, the trace of $v$ must be equal to $c_k$ a.e. on $\partial Z_j \cap \partial Z_k$. Therefore, the only way for $c_j y + V_j$ to remain constant on a set of non-null $\H^2$-measure is to have $c_j = 0$.
\par
In both cases this implies that $\nabla v^n$ converges strongly to $0$ in $L^2(Z_j;\mathbb{R}^{3\times 3})$, and thus
\begin{equation}\label{lim0inJ}
\lim_{n\to\infty} \sum_{j\in J}\int_{Z_j}\sum_{r,q=1,2,3} (\partial_q v_r^n)^2 dy = 0.
\end{equation}

Finally, take $i\in I$. By \eqref{ineqInI} we have
\begin{multline*}
\int_{Z_i}\left[ P(y;\partial_1 v_1^n,\partial_2 v_2^n,\partial_3 v_3^n) + \sum_{r<q} Q(y;\partial_q v_r^n,\partial_r v_q^n) \right]dy \ge \\ 
\alpha \int_{Z_i} \left[ 
	(\partial_1 v_1^n + \partial_2 v_2^n + \partial_3 v_3^n)^2 + (\partial_2 v_1^n + \partial_1 v_2^n)^2 + (\partial_3 v_1^n + \partial_1 v_3^n)^2 + (\partial_3 v_2^n + \partial_2 v_3^n)^2
 \right]dy.
\end{multline*}
By virtue of \eqref{tendsToZeroY3} we also have
\begin{equation}\label{tendsToZeroZi}
\int_{Z_i} \left[ 
	(\partial_1 v_1^n + \partial_2 v_2^n + \partial_3 v_3^n)^2 + (\partial_2 v_1^n + \partial_1 v_2^n)^2 + (\partial_3 v_1^n + \partial_1 v_3^n)^2 + (\partial_3 v_2^n + \partial_2 v_3^n)^2
 \right]dy \to 0\quad\mbox{as } n\to\infty.
\end{equation}
Limits \eqref{lim0inJ}, \eqref{lim0inK} combined with \eqref{null-Lagr} yield
\begin{equation*}
\lim_{n\to\infty} \sum_{i\in I} \int_{Z_i} \sum_{r=1}^3 \adj_{rr}(\nabla v^n)dy = 0.
\end{equation*}
Therefore, upon subtracting this quantity to the sum over $i\in I$ of \eqref{tendsToZeroZi} we conclude that
\begin{equation}\label{lim0inI}
\lim_{n\to\infty} \sum_{i\in I} \int_{Z_i} \sum_{r,q = 1}^3 (\partial_q v_r^n)^2 dy = 0.
\end{equation}
Finally, limits \eqref{lim0inJ}, \eqref{lim0inK} and \eqref{lim0inI} contradict condition \eqref{hipRedAbs}. The proof is thus complete.
\end{proof}

\par\bigskip\noindent
{\bf Acknowledgments.} The authors wish to thank A. Braides for the helpful Theorem~\ref{BBResult}.
\par\noindent
This work has been partially supported by the project MTM2011-24457 of the  ``Ministerio de Econom\'{\i}a y Competitividad" of Spain, and the project FQM-309 of the ``Junta de Andaluc\'{\i}a". A.P.-M. is a fellow of ``Programa de FPI del Ministerio de Econom\'ia y Competitividad" reference BES-2012-055158. A.P.-M. is grateful to the {\em Institut de Math\'ematiques Appliqu\'ees de Rennes} for its hospitality, where this work was carried out during his stay March 2 - June 29 2015.

\end{document}